\documentclass[12pt, oneside, letterpaper]{amsart}
\pdfoutput=1
\usepackage{ifpdf}
\usepackage[british]{babel}

\usepackage[longnamesfirst]{natbib}
\setcitestyle{round}

\usepackage{amsmath, amsthm, amsfonts, euscript}

\usepackage{color,graphicx}

\usepackage[left=2cm,right=2cm]{geometry}

\usepackage{verbatim}

\ifpdf
 \usepackage{xcolor}
 \input dvipsnam.def
\fi
\definecolor{linkcolor}{named}{Maroon}
\definecolor{citecolor}{named}{OliveGreen}
\definecolor{urlcolor}{named}{RoyalPurple}
\definecolor{okcolor}{named}{OliveGreen}
\definecolor{alertcolor}{named}{BrickRed}

\usepackage[colorlinks,bookmarksopen]{hyperref}
 \hypersetup{citecolor=citecolor}
 \hypersetup{linkcolor=linkcolor}
 \hypersetup{urlcolor=urlcolor}
 \hypersetup{
   pdfauthor={Maury Bramson, Krzysztof Burdzy and Wilfrid Kendall},
   pdftitle={Rubber Bands, Pursuit Games and Shy Couplings}
 }

\newcommand{\ball}{\operatorname{\mathcal B}}

\def\ol{\overline}
\def\prt{\partial}

\renewcommand{\d}{\operatorname{d}}
\newcommand{\dist}{\operatorname{dist}}

\newcommand{\intrinsicdist}{\operatorname{dist}_\text{\bf I}}

\newcommand{\CAT}[1]{\text{CAT}({#1})}

\newcommand{\Prob}[1]{\operatorname{\mathbb{P}}\left[#1\right]}
\newcommand{\Reals}{\mathbb{R}}

\def\prt{\partial}
\def\wt{\widetilde}
\def\wh{\widehat}
\def\bH{{\bf H}}
\def\bK{{\bf K}}
\def\NN{{\mathcal N}}

\def\SS{{\mathcal S}}

\def\DD{{\mathcal D}}
\def\CC{{\mathcal C}}
\def\eps{\varepsilon}
\def\JJ{\mathbb{J}}
\def\KK{\mathbb{K}}

\def\P{\mathbb{P}}
\def\al{\alpha}
\def\idist{\intrinsicdist}

\newtheorem{thm}{Theorem}[section]
\newtheorem{lemma}[thm]{Lemma}
\newtheorem{lem}[thm]{Lemma}
\newtheorem{prop}[thm]{Proposition}
\newtheorem{cor}[thm]{Corollary}
\newtheorem{defn}[thm]{Definition}

\newtheorem{example}[thm]{Example}

\newtheorem{conjecture}[thm]{Conjecture}

\newtheorem{rem}[thm]{Remark}
\numberwithin{equation}{section}

\def\lp{K}
\def\Lp{{\mathcal K}}

\def\includegraphicsKB{\includegraphics}

\begin{document}

\title{Rubber Bands, Pursuit Games and Shy Couplings}

\author{Maury Bramson}
\author{Krzysztof Burdzy}
\author{Wilfrid Kendall}

\address[M.B.]{School of Mathematics, Vincent Hall, 206 Church St. SE., University of Minnesota, Minneapolis, MN 55455, USA}
\email{bramson@math.umn.edu}
\address[K.B.]{Department of Mathematics, Box 354350, University of Washington, Seattle, WA 98195, USA}
\email{burdzy@math.washington.edu}
\address[W.S.K.]{Department of Statistics, University of Warwick, Coventry CV4 7AL, UK}
\email{w.s.kendall@warwick.ac.uk}

\thanks{Research supported in part by NSF Grants  DMS-0906743 and DMS-1105668, and
by grant N N201 397137, MNiSW, Poland. }



\begin{abstract}\noindent
In this paper,
we consider pursuit-evasion and probabilistic consequences of some geometric notions for bounded and suitably regular domains in Euclidean space that are \(\CAT\kappa\) for some $\kappa > 0$. These geometric notions are useful for analyzing the related problems of (a) existence/nonexistence of successful evasion strategies for the Man in Lion and Man problems, and (b) existence/nonexistence of shy couplings for reflected Brownian motions.
They involve properties of \emph{rubber bands} and 
the extent to which
a loop in the domain in question can be 
deformed to a point without, in between,
increasing its loop length. The existence of a \emph{stable rubber band} will imply the existence of a successful evasion strategy but, 
if all loops in the domain are \emph{well-contractible}, then no successful evasion strategy  will exist and  there can be no co-adapted shy coupling.  
For example, there can be no shy couplings in bounded and suitably regular star-shaped domains and so, in this setting, 
any two reflected Brownian motions must 
almost surely make arbitrarily close encounters as $t\rightarrow\infty$.
\end{abstract}

\maketitle

\section{Introduction}\label{sec:intro}

The motivation for this article is a conjecture
about shy couplings, that is, about constructions of pairs of reflected Brownian motions in a bounded Euclidean domain 
that are contrived
so that, 
for some 
fixed \(\eps>0\),
they never come within distance $\eps$ 
of
each other. In \citet{BramsonBurdzyKendall-2011}, we showed that strong results about nonexistence of shy couplings could be proved using ideas of pursuit-evasion games and modern metric geometry.
In the current paper, we introduce new metric geometry notions (such as``rubber bands'' and ``well-contractible loops") 
that can be used to derive general results about pursuit-evasion games and further results about shy coupling.
In particular, while \citet{BramsonBurdzyKendall-2011} shows that shy couplings cannot be supported by suitably regular bounded \(\CAT0\) domains, here we show that shy couplings
cannot be supported by a substantially larger family of domains
including, for example, bounded star-shaped domains with suitably regular boundaries
(see Definition \ref{def:CAT-definitions} for the definitions of \(\CAT0\) and \(\CAT\kappa\)
domains). 
Our results apply to domains $D \subset \mathbb{R}^d$, for $d\ge 2$, but their main
interest is in $d\ge 3$, since all bounded simply connected domains in $d=2$ are $\CAT0$, and hence the
results from \citet{BramsonBurdzyKendall-2011} apply in that setting.

We first summarize our results for pursuit-evasion games.  In this deterministic setting, there are two players, a Lion and a Man,
each of whom is constrained to remain in a given bounded domain $D$.  Both the Lion and the Man are allowed 
to move within $D$ at up to unit speed.  We are interested in the question as to whether, 
for some strategy of the Lion, the Lion
is able to come within distance $\varepsilon$ of the Man,  irrespective of the strategy of the Man and
for any $\varepsilon >0$.  
We will say that the Lion \emph{captures} the Man or the Man \emph{evades} the Lion, depending on whether or not such a strategy exists for every pair of initial positions.

The pursuit-evasion problem in a disk is a well-known problem, and includes the question
as to whether the Man can avoid the Lion indefinitely (even though the distance between them
is allowed to go to 0).  See, for example,
\cite{Isaacs-1965}, \cite{Littlewood-1986}, \cite{Nahin-2007}.
In our current setting, we consider bounded domains $D \subset \mathbb{R}^d$.

For the Lion and Man pursuit-evasion problem, we will determine conditions on the domain $D$ under which
the Man can evade the Lion and under which the Lion can capture the Man.  Under suitable
side conditions, the first scenario holds when $D$ possesses
a \emph{stable rubber band}, which is, in essence, a locally distance-minimizing loop.  
Section 3 is devoted to showing this, with the main result being Theorem \ref{thm:srb}.  
The second scenario holds when all loops in $D$ are \emph{well-contractible}, 
which in essence means that the loop can be 
contracted to a point, with the length of the intermediate loops decreasing at a uniform rate
with respect to the homotopy parameter.
Section 4 shows that the Lion is able to capture the Man when all loops are well-contractible, 
with the main result being Theorem \ref{thmo13}.  

The assumption that $D$ is $\CAT\kappa$ figures prominently in both arguments, and
in the succeeding sections of the paper.  
Roughly speaking, a domain $D$ satisfies the \(\CAT\kappa\) condition if suitably small 
triangles defined using the  intrinsic distance in $D$ have angles no greater than angles of triangles with the same side lengths on the surface of the Euclidean
sphere of radius \(1/\sqrt\kappa\) (the formal definition of \(\CAT\kappa\) domains will be given later in the paper).
We will also require some regularity on the boundary of $D$, which will be given by the
uniform exterior sphere and uniform interior cone conditions
(see Definitions \ref{def:exterior-sphere-condition}-\ref{def:UICC}); a domain $D$ satisfying both
conditions will be referred to as an ESIC domain.  
An ESIC domain whose loops are all well-contractible will be referred 
to as a \emph{CL domain}.  Since an ESIC domain is $\CAT\kappa$, for some
$\kappa \ge 0$ (see Corollary \ref{n7.2}), these two boundary conditions will in fact suffice for many of our results.  
The definitions of these
terms and others that will be employed in the paper are given in Section 2.

The second half of the paper is devoted mostly to shy couplings.  A reflected 
Brownian motion on a domain $D$
is said to admit a \emph{shy coupling} if there exists a coupling of Brownian motions $X$ and $Y$
on $D$, for some choice of initial points \(x\) and \(y\), such that
\[
 \Prob{\inf \left\{\dist(X_t, Y_t) : {0\leq t<\infty}\right\}\,>\,0\;|\; X_0=x, Y_0=y} \quad>\quad0\,.
\]
(We consider throughout only couplings that are co-adapted, that is, that do not anticipate
the future.)  An example of a shy coupling is given by Brownian motions $X$ and $Y$ on a circle,
where $Y$ is produced from $X$ by a nontrivial rotation.  Except for similar specialized examples, all
known results involve the absence of shy couplings, and only a partial theory is known.
\citet{BenjaminiBurdzyChen-2007}, who introduced the notion of shy coupling,  showed
that no shy couplings exist for reflected Brownian motion in convex bounded
planar domains with $C^2$ boundaries containing no line segments;
\citet{Kendall-2009a} used a direct and somewhat quantitative approach to remove regularity requirements in the convex case.
\citet{BramsonBurdzyKendall-2011} showed that no
shy couplings exist for bounded ESIC domains that are $\CAT0$.  
(Also see \citet{BramsonBurdzyKendall-2011} for further background.)

Section 5 extends the approach taken in \citet{BramsonBurdzyKendall-2011}, and shows, 
in Theorem \ref{thmo24}, that no shy couplings exist for bounded CL domains. 
The basic idea behind the argument
is to transform the process of coupled Brownian motions, by using the 
Cameron-Martin-Girsanov transformation and scaling time, 
to a process where each sample path is approximated by a solution of the Lion and
Man problem.  In the present context, one can then apply  
Theorem \ref{thmo13} to this Lion and Man problem. 

In Section 6, it is shown that there is no analogous application of Theorem \ref{thm:srb}
whereby the existence of a shy coupling follows from the existence of a stable 
rubber band.  In fact, starting with any bounded domain possessing a stable rubber band,
it is possible to append another larger domain, which preserves the rubber band, 
so that the combined domain has no shy couplings.

A number of examples of CL domains and domains with rubber bands are given in Section 7. 
In particular, in Examples \ref{o24.1.ii} -\ref{jan9.3}, various examples of CL domains are
given, such as restrictions of $\CAT0$ domains that themselves are not $\CAT0$, 
including star-shaped domains.  
At the end of the section, we conjecture that, off a nowhere dense family of domains (taken with respect to the Gromov-Hausdorff distance), 
all bounded domains with bounded principal curvatures are either CL or possess
a \emph{semi-stable rubber band} (that is, the rubber band is minimal, but not 
necessarily strictly minimal).  

Employing a result from \citet{Lytchak-2004}, the claim that ESIC domains are
$\CAT\kappa$, for some $\kappa \ge 0$, is shown in the short appendix.

\section{Rubber bands}\label{sec:rubber}

In this section, we introduce some basic notions for domains in Euclidean space, including: conditions for suitable regularity of the boundary, 
intrinsic distance and related concepts from metric geometry, and rectifiable loops and their homotopies. 
Most importantly, we introduce the new notion of \emph{rubber bands}, as well as several associated concepts.
The notion of rubber band will play a key r\^ole in the main results in later sections on pursuit-evasion and on shy coupling of reflected Brownian motion.   

Suppose that $D\in \Reals^d$ is a bounded domain (that is, an open connected set).
The \emph{intrinsic distance} $\idist(v,z)$ between $v,z\in D$ is the infimum of 
lengths $\ell_\Gamma$ of rectifiable arcs $\Gamma \subset D$ that 
contain $v$ and $z$. We will typically wish
to restrict our attention to domains for which the notion of intrinsic distance extends to the entire closure \(\ol D\) without discontinuity at the boundary \(\partial{D}\).
To achieve this, we 
follow \citet{BramsonBurdzyKendall-2011} in requiring that \(D\) satisfy both
the uniform exterior sphere condition and the uniform interior cone condition defined below.  Here and elsewhere,
$\ball(z, r)$ denotes the open Euclidean ball of radius $r$ centered at $z$. 

\begin{defn}[Uniform exterior sphere condition, from {\citealp[\S1, Condition $(A)$]{Saisho-1987}}]\label{def:exterior-sphere-condition}
A domain \(D\) satisfies a \emph{uniform exterior sphere condition
based on radius \(r\)} if, for every \(z\in \partial D\), the set of ``exterior normals''
\(\NN_{z,r} = \{\nu \in \Reals^d: |\nu|=1, \ball(z+ r \nu, r) \cap D = \emptyset\}\)
is non-empty, with \(\NN_{z,r}=\NN_{z,s}\) for \(0<s\leq r\).
\end{defn}%
\begin{defn}[Uniform interior cone condition, from {\citealp[\S1, Condition $(B^\prime)$]{Saisho-1987}}]\label{def:UICC}
A dom\-ain
\(D\) satisfies a \emph{uniform interior cone
condition based on radius \(\delta>0\) and angle
\(\alpha\in(0,\pi/2]\)} if, for every \(v\in
\partial D\), there is at least one unit vector \(\mathbf{m}\) such that the cone
\(C(\mathbf{m})=\{z:\langle z,\mathbf{m}\rangle > |z|\cos\alpha\}\) satisfies
\[
\left(w + C(\mathbf{m})\right)\cap\ball(v,\delta) \quad\subseteq\quad D \qquad
\text{ for all } w\in D\cap\ball(v,\delta)\,.
\]
We say that
 the cone \(w+C(\mathbf{m})\)
 \emph{is
based on \(w\)
 and angle \(\alpha\in(0,\pi/2]\)}.
\end{defn}
It was shown in \citet[Section 2]{BramsonBurdzyKendall-2011} that the uniform interior cone condition is equivalent to the better known Lipschitz boundary condition
(see Definition \ref{def:lipschitz-domain}). 

The uniform exterior sphere and uniform interior cone conditions were employed by \citet{Saisho-1987} 
to define reflecting Brownian motion in \(D\).
However, the conditions are also useful in establishing regularity of the intrinsic distance.
In particular, if \(D\) satisfies both conditions, then the intrinsic distance between two close points in \(D\) is comparable to the Euclidean distance
\cite[Proposition 12]{BramsonBurdzyKendall-2011}, and  the 
intrinsic distance therefore extends to the entire closure \(\ol D\) without discontinuity at \(\partial{D}\).

The following two simple examples demonstrate the need for both conditions: 
\begin{example}\label{ex:slit-disk}
Suppose that \(D\) is formed from  the disc $\ball((0,0),1)$ by deleting the line segment from $(0,0)$ to $(1,0)$. Then \(D\)
satisfies the uniform interior cone condition, although
the uniform exterior sphere condition fails on the line segment from $(0,0)$ to $(1,0)$.
The intrinsic distance cannot be extended to $\ol D$ in a continuous manner.
\end{example}
\begin{example}\label{ex:eroded-cube}
Suppose that \(D\) is formed from the cube \([-1,1]^3\), in \(3\)-space, by deleting the two continuous families of closed balls
\(\{\ol{\ball((1,0,u),1)}:-1/2\leq u\leq 1/2\}\) and \(\{\ol{\ball((-1,0,u),1)}:-1/2\leq u\leq 1/2\}\). Here, \(D\)
satisfies the uniform exterior sphere condition, although the uniform interior cone condition fails at the open line
segment \(\{(0,0,u):-1/2<u<1/2\}\). The domain \(D\) is connected, with the two points \((0,\pm\eps,0)\) 
being distance \(\sqrt{1+4\eps^2}\) apart with respect to the
intrinsic distance for \(D\). On the other hand, the two points
are distance \(2\eps\) apart
in terms
of  both the Euclidean metric and the intrinsic distance for \(\ol D\).
Thus, the intrinsic distance cannot be extended to $\ol D$ in a continuous manner. 
\end{example}
We therefore typically  consider domains that satisfy the uniform exterior sphere and interior cone conditions;
we refer to such domains as \emph{ESIC domains} (i.e., uniform Exterior Sphere and Interior Cone domains).
(In principle, one might consider generalizing the following results to non-ESIC domains; one then needs to 
 take into account  the pathologies illustrated in the two preceding examples.)


The following classic curvature comparison property is central to our arguments.
Following \citet[\S II.1, Definition 1.1]{BridsonHaefliger-1999}
we define the $\CAT\kappa$ property as follows.
\begin{defn}\label{def:CAT-definitions}
For \(\kappa>0\), the domain $D$ is a \(\CAT\kappa\)
domain if any two distinct points with distance less than $\pi/\sqrt{\kappa}$ are joined by a geodesic and the distance between any two
points on the perimeter of any geodesic triangle $\triangle pqr$ of perimeter less than $2\pi/\sqrt{\kappa}$
is no greater than the distance
between the corresponding points of the model triangle $\triangle \wt p\, \wt q\, \wt r$ with the same side lengths on
the \(2\)-dimensional Euclidean sphere of radius $1/\sqrt{\kappa}$. 
The domain $D$ is a \(\CAT0\)
domain if any two distinct points at whatever distance are joined by a geodesic and the distance between any two
points on the perimeter of any geodesic triangle $\triangle pqr$
is no greater than the distance
between the corresponding points of the model triangle $\triangle \wt p\, \wt q\, \wt r$ with the same side lengths in
the \(2\)-dimensional Euclidean plane.
\end{defn}

A bounded domain satisfying the uniform exterior sphere and uniform interior 
cone conditions is 
\(\CAT\kappa\), for some \(\kappa>0\). 
We sketch a proof in in Appendix \ref{app:reach}. The claim has already been proved in the literature in a slightly weaker form (see Remark \ref{m29.1}).
From time to time in the article, we will explicitly recall that ESIC domains satisfy the \(\CAT\kappa\) 
property, since our estimates often make use of
the curvature parameter \(\kappa\).

For $\kappa>0$, the scaling $D \to \sqrt{\kappa} D$ transforms a $\CAT{\kappa}$ domain into 
a $\CAT{1}$ domain. (See, for example, the appendix to \citealp{AlexanderBishopGhrist-2009}.)
Note that, for $\kappa_1 \le \kappa_2$, if a domain $D$ is $\CAT{\kappa_1}$, then it is also
automatically $\CAT{\kappa_2}$.
Where convenient, we will limit our arguments to the cases $\kappa=0,1$.



We next introduce some notation for rectifiable loops and the concatenation
of curves in \(\ol D\). 
Let $\SS$ be the circle with radius $1$ centered at the origin; it will be convenient to identify $\SS$ 
with
$\{e^{2\pi iu}, 0\leq u < 1\}$.
Let $\Lp$ be the family of all loops $\lp$ in $\ol D$ with finite length, 
i.e., $\lp: \SS\to \ol D$ is a continuous mapping, with $\lp(\SS)$ being 
rectifiable with length $\ell_\lp <\infty$.
We will reparametrize $\lp$ by its length measured from a base point \(\lp(0)\), i.e.,
$\lp = \{\lp(t): t\in[0,\ell_\lp)\}$ such that, for every $s\in[0,\ell_\lp)$, the length
of $\{\lp(t): t\in[0,s]\}$ is $s$.
Accordingly, we may view any loop \(\lp\in\Lp\) as a Lipschitz closed curve with Lipschitz constant \(1\).
The same conventions about parametrization by length
will apply to other rectifiable curves that are not necessarily loops.
For convenience, we will sometimes abuse notation by writing $\lp$ instead of $\lp(\SS)$, for example, 
writing $\lp \subset D$.
For $\lp\in \Lp$, we define the Euclidean tubular neighbourhood \(\ball(\lp,r)\) of \(\lp\) by $\ball(\lp,r) = \{z\in \ol D: \dist(z, \lp) < r\}$.
(Recall that \(\ball(z,r)\) 
denotes the open Euclidean ball of radius \(r\) centered on \(z\).)

The \emph{concatenation} \(f*g\) of curves
\(f:[0,T]\to\Reals^d\) and \(g:[0,S]\to\Reals^d\),
with \(f(T)=g(0)\), is the curve \(f*g : [0,T+S]\to\Reals^d\),
\[
(f*g)(u)\quad=\quad\
\begin{cases}
 f(u) & \text{ if } 0\leq u\leq T\,,\\
 g(u-T) & \text{ if } T < u \leq T+S\,.
\end{cases} 
\]
We write \(f^{-1}\) for the reversed curve \(t\mapsto f(T-t)\).
If \(f(0)=f(T)\), then we write \(f^{*n}\) for the \(n\)-fold concatenation of \(f\) with itself, for \(n=1, 2, \ldots\); 
in particular,
for a  loop $\lp \in \Lp$ and \(n\) a positive integer, the \(n\)-fold \emph{concatenation power} $\lp^{*n}\in \Lp$ satisfies the conditions $\ell_{\lp^{*n}} = n\ell_\lp$ 
and $\lp^{*n} (t)= \lp(t\!\!\mod\ell_\lp)$ for $t\in[0,n\ell_\lp)$.
If \(n=-m\) is negative, then we define $\lp^{*n}\in \Lp$ to be the reversal of \(\lp^{*m}\).




The intrinsic Hausdorff distance between $A,B\subset \Reals^d$ is defined by
\begin{align*}
d_H(A,B) = \max \left\{
\sup_{v\in A} \inf_{z\in B} \idist(v,z)\,,\,\,\,
\sup_{v\in B} \inf_{z\in A} \idist(v,z)
\right\}.
\end{align*}
We will use intrinsic Hausdorff distance to measure distance between loops viewed as subsets of the closure \(\ol D\) of
the domain \(D\).

It will be important to identify instances in which loops can be contracted to points, 
to identify other instances
in which loops cannot be contracted at all, and to distinguish between weak contractions 
as opposed to contractions for which
contraction occurs at least at a uniform rate. (We consider only ESIC domains in order to avoid needing
to consider the kind of boundary issues illustrated by Examples \ref{ex:slit-disk}, \ref{ex:eroded-cube}.)
\begin{defn}\label{def:rubband}
Suppose $D\subset \Reals^d$, with $d\ge 2$.
\begin{itemize}
\item[(a)] 
A
loop $\lp\in\Lp$ is 
a \emph{contractible loop}
if there exists a length-monotonic homotopy of \(\lp\) with a point \(z\in \ol D\), namely,
a continuous mapping $H:\SS\times [0,1]\to \ol D$ such that
\begin{enumerate}
\item For every $\gamma \in[0,1)$, there exists $\lp_\gamma \in \Lp$ such that 
$H(e^{2\pi i t},\gamma)= \lp_\gamma(t\ell_{\lp_\gamma})$, for $t\in[0,1)$.
\item $\lp_0 = \lp$.
\item $H(\SS,1) = \lp_1 = \{z\}$ for the specified $z\in\ol D$.
\item The function $\gamma \to \ell_{\lp_\gamma}$ is non-increasing on $[0,1]$.
\end{enumerate}
We will identify $\lp(\gamma,t)$ with the family $\{K_\gamma\}_{\gamma \in [0,1)}$ 
and call it a \emph{contraction} of $\lp$.

\item[(b)] 
A contractible loop $\lp\in \Lp$ is 
\emph{well-contractible, with contractibility constant \(c\in(0,\infty)\)}, if 
there exists a length-monotonic homotopy
contraction $\{\lp_\gamma\}_{\gamma \in [0,1)}$ such that, for 
all $0\leq\gamma< \eta\leq 1$,
$$
\ell_{\lp_\gamma} - \ell_{\lp_\eta}
\quad\geq\quad c \;d_H(\lp_\gamma, \lp_\eta)\;\ell_{\lp_\gamma}.
$$
In words, this says that the homotopy can be chosen so that the relative rate of contraction is bounded away from zero when measured using the change in the Hausdorff distance.
Note that the contractibility constant $c$ may depend on the point $H(\SS,1)$ to which
the loop is contracted.

\item[(c)] 
A bounded ESIC domain \(D\) is a \emph{contractible loop (CL) domain} 
if there exists a constant \(c>0\) such that, for each $\lp \in \Lp$, there exists $z\in \ol D$ 
such that $\lp$ is well-contractible to $z$ with the contractibility constant \(c\).
(We can then also say that the loops in \(D\) are \emph{uniformly contractible}.)
\end{itemize}
\end{defn}
\begin{rem}\label{rem:CL-domain}
The definitions of contractible loops and well-contractible loops apply to loops in any set 
$D\subset\Reals^d$ but we will limit our considerations to ESIC domains $D$ because the behavior of such loops may be strange in non-ESIC domains.
\end{rem}

We introduce the following concepts when the loop length-functional is at a ``local minimum".
\begin{defn}\label{o19.1}
\begin{itemize}
 \item[(a)]  A loop $\lp\in\Lp$ is a \emph{semi-stable rubber band} if, 
for some $\eps>0$, the following holds: Suppose that $\lp_1\in \Lp$ and there exists 
a continuous mapping $H: \SS \times [0,1] \to \overline{\ball(\lp,\eps)}$ such that $H(e^{2\pi i t},0)= \lp(t\ell_\lp)$ for $t\in[0,1)$ and $H(e^{2\pi i t},1)=  \lp_1(t\ell_{\lp_1})$ for $t\in[0,1)$. 
Then $\ell_{\lp_1}\geq \ell_\lp$. 
\item[(b)]
A loop $\lp\in\Lp$ is 
a \emph{stable rubber band} if
it is semi-stable and if,
for some $\eps>0$ 
and all  $0<\eta\leq\eps$, there exists $\delta=\delta(\eta,\eps)>0$ 
such that the following holds:
Suppose that $\lp_1\in \Lp$, $d_H(\lp, \lp_1)\geq \eta$ and, for some $n\geq 1$, there exists 
a continuous mapping $H: \SS \times [0,1] \to \overline{\ball(\lp,\eps)}$ such that
$H(e^{2\pi i t},0)= \lp(tn\ell_\lp)$ for $t\in[0,1)$ and $H(e^{2\pi i t},1)= \lp_1(t\ell_{\lp_1})$ for $t\in[0,1)$.
Then $\ell_{\lp_1} >n\ell_\lp + \delta$.  
In words, if a concatenation power $\lp^{*n}$
of \(\lp\) with $n\ne0$ can be locally
perturbed to a loop $K_1$, 
then $K_1$ must be longer than $\lp^{*n}$ by at least an amount depending on the intrinsic Hausdorff distance between the two loops.
\end{itemize}
\end{defn}

As noted above, an ESIC domain \(D\) must be \(\CAT\kappa\), for some \(\kappa\geq0\).
We conclude this section with two lemmas that employ the \(\CAT\kappa\) property, followed 
by  a pair of remarks.
The first lemma shows that, in ESIC domains, any two rectifiable loops that 
are suitably close to each other
are also connected by a (not necessarily length-monotonic) local homotopy.
We adopt the convention that \(\pi/\sqrt{\kappa}=\infty\) if \(\kappa=0\),
in order to avoid needing to distinguish between \(\kappa=0\) and \(\kappa>0\).
\begin{lem}\label{lem:cat-kappa-homotopy}
 Let $D$ be an ESIC domain that is $\CAT\kappa$, with \(\kappa\geq0\).
Suppose that \(\lp_0\), \(\lp_1\) are rectifiable loops such that 
\[
 \idist(\lp_0(t \ell_{\lp_0}), \lp_1(t\ell_{\lp_1})) \quad\leq\quad \eps
\]
for all \(0\leq t\leq 1\), for some \(\eps<\pi/\sqrt{\kappa}\). Then \(\lp_0\) and \(\lp_1\) are homotopic within \(\ball(\lp_0,\eps)\).
\end{lem}
\begin{proof}
First note that it follows from Definition \ref{def:CAT-definitions} that any geodesic of total length less than \(\pi/\sqrt{\kappa}\) is uniquely defined
by its end-points, is minimal, and depends continuously on its end-points. (This dependence is uniform in case the total length is bounded away from \(\pi/\sqrt{\kappa}\).)

We define the homotopy \(H:[0,1]^2\to\ball(\lp_0,\eps)\) by
\[
 H(s,t)\quad=\quad \gamma^{(t)}(s\ell_{\gamma^{(t)}})\,,
\]
where \(\gamma^{(t)}\) is the unit-speed geodesic from \(\lp_0(t \ell_{\lp_0})\) to \(\lp_1(t\ell_{\lp_1})\).
The continuity  of $H(\cdot,\cdot)$ follows directly from the properties in the first paragraph of the proof. 
\end{proof}

We can employ the previous lemma to show that a semi-stable rubber band is locally geodesic.

\begin{lem}\label{lem:locally-geodesic}
If $\lp$ is a semi-stable rubber band
in an ESIC domain $D$, 
then it is locally geodesic in the intrinsic distance metric.
\end{lem}
\begin{proof}
The loop \(\lp\) is locally geodesic in the intrinsic distance metric if, for some \(\eps>0\) 
and any \(0\leq s<t<\ell_\lp\), (i) when \(t-s<\eps/2\), then
\(\{\lp(v): s\leq v\leq t\}\) determines a length-minimizing intrinsic geodesic from \(\lp(s)\) to \(\lp(t)\) and (ii) when \((\ell_K -t) + (s-0) =\ell_\lp-t+s<\eps/2\), then
\(\{\lp(v): t\leq v<\ell_\lp\}\) followed by \(\{\lp(v): 0\leq v\leq s\}\) determines a length-minimizing intrinsic geodesic from \(\lp(t)\) to \(\lp(s)\).

We will demonstrate case (i); a similar argument holds for case (ii).
First note that $D$ must be $\CAT\kappa$ for some $\kappa>0$.  
Choose $\eps>0$ as in Definition \ref{o19.1}(a)
so that $\eps<\pi/\sqrt{\kappa}$. 
Suppose \(0\leq s<t<\ell_K\) and \(t-s<\eps/2\).  Then \(\lp(v)\in\ball(\lp(s),\eps/2)\) for \(s\leq v\leq t\), because \(\lp\) has Lipschitz constant \(1\).
Were \(\{\lp(v): s\leq v\leq t\}\) not length-minimizing, then it would be possible to replace this section of the loop by a strictly shorter segment,
thus producing a new loop \(\lp_1\) with strictly smaller total length. Moreover, 
by the triangle inequality, 
$\idist(\lp(v), \lp_1(v)) \leq \eps$ for $s\leq v \leq t$.
Since
 $\lp(v)= \lp_1(v)$ for $v\notin [s,t]$, we have
$\idist(\lp(v), \lp_1(v)) \leq \eps$ for all $v$.
Hence, by Lemma \ref{lem:cat-kappa-homotopy}, it follows that $\lp$ and $\lp_1$ are homotopic within 
$\ball(\lp,\eps)$. This contradicts the assertion that \(\lp\) is semi-stable, and therefore implies
that the segment \(\{\lp(v): s\leq v\leq t\}\) must be length-minimizing, and hence is a minimal geodesic.
\end{proof}

\begin{rem}\rm
At the intuitive level,
a rubber band is almost the same as a non-constant harmonic map from 
a circle to a closed set in the Euclidean space, or in other words a closed geodesic. 
However, the theory of harmonic maps does not seem to be 
relevant to our study. 
(The literature on harmonic maps is huge.
Succinct summaries of the general theory of smooth harmonic maps can be found in \citet{EellsLemaire-1978,EellsLemaire-1988};
see also the monograph by
\citet{harmonic}.
Non-smooth harmonic maps are discussed in \citet{EellsFuglede-2001}.)
\end{rem}

\begin{rem}\label{rem:stability}\rm
 Note that the property of \(\lp\) being a stable rubber band,
respectively a semi-stable rubber band,
in a domain \(D\) is \emph{local} to \(\lp\), in the sense that \(\lp\) remains stable, respectively semi-stable,
if the domain \(D\) is altered, as long as \(D\cap \ball(\lp,\eps)\) is not altered for some \(\eps > 0\).
\end{rem}


\section{Domains with stable rubber bands}\label{sec:stable}

In this section, we analyze domains that contain stable rubber bands.  In Definition \ref{LM2}, we 
formulate the Lion and Man problem, and specify
what it means for the Man to have a successful evasion strategy.  Theorem \ref{thm:srb}
is the main result of this section, where we will show that, for ESIC  domains containing a
stable rubber band, there is always a successful evasion strategy for the Man.  
The property that any ESIC domain is $\CAT\kappa$, for some $\kappa \ge 0$, will be
employed repeatedly.  

%
%

We begin by establishing a mathematical framework for
pursuit and evasion.
In Definition \ref{LM2}, the
path of the Man is represented by a continuous curve $y(t)$
and that of the Lion by a continuous curve $x(t)$. Here, $\Reals_+ = \{t\in \Reals: t\geq 0\}$ 
and $\CC = C(\Reals_+, \ol D)$ is the space of continuous functions 
on $\Reals_+$ with values in $\ol D$.

\begin{defn}\label{LM2}
Suppose that \(D\) is an ESIC domain.

(i) $\{x(t), t\geq 0\}$ is an \emph{admissible curve}
if it is continuous,
locally rectifiable and parametrized so that
$\idist(x(s),x(t)) \leq |s-t|$ for all $s,t\geq 0$.
Note that this implies $x'(t)$ exists for almost all $t\geq 0$ and $x(t) -
x(0) = \int_0^t x'(s) ds$ for every $t\geq 0$.

(ii) Let $\Lambda$ be the family of all quadruples $(x,
y, F_x, F_y)$ 
such that $x$ and $y$ are admissible curves and $x,
y, F_x $ and $F_y$
satisfy the following properties. The
functions $F_x: \Reals_+\times \CC^2 \to \Reals^d$ and
$F_y: \Reals_+\times\CC^2 \to \Reals^d$ are measurable and
such that $x'(t) = F_x(t,x(\,\cdot\,),
y(\,\cdot\,))$ for all $t$ where $x'(t)$ exists
and, similarly, $y'(t) = F_y(t,x(\,\cdot\,),
y(\,\cdot\,))$ for all $t$ where $y'(t)$ exists.
Moreover, $F_x$ and $F_y$ are non-anticipative in the
sense that, if $x, x^*,y,y^* \in \CC$, $x(s)
= x^*(s) $ and $y(s) = y^*(s) $ for $s\leq t$, then
$F_x(t,x(\,\cdot\,), y(\,\cdot\,))=
F_x(t,x^*(\,\cdot\,), y^*(\,\cdot\,))$; the analogous
condition is satisfied by $F_y$.

(iii) \emph{The Man has a successful evasion strategy} if, for some pair 
$x_0, y_0\in\ol D$,
(a) There exists $(x, y,
F_x, F_y)\in \Lambda$, with $x(0) = x_0$ and $y(0) = y_0$. 
(b) 
Suppose that 
$F_x$ and $x$, with $x(0) = x_0$, are such that 
there exist $y$ and $F_y$ with $y(0)=y_0$ and $(x, y,
F_x, F_y)\in \Lambda$.
Then 
there exist $y$ and $F_y$ such that $y(0)=y_0$, $(x, y,
F_x, F_y)\in \Lambda$
and the evasion condition $\inf_{t\in[0,\infty)}
\intrinsicdist(x(t),y(t))
>0$ holds.

(iv) Conversely,  \emph{there is no successful evasion strategy for
the Man} (or that \emph{the Lion can capture the Man)}
if, for each pair $x_0, y_0\in\ol D$, with $x_0\neq y_0$,  and every
$F_y$ and $y$ with $y(0) = y_0$, with at least one tuple $(x, y, F_x, F_y)\in
\Lambda$ satisfying $x(0) =
x_0$, there exist $x$ and $F_x$ with  $(x, y,
F_x, F_y)\in \Lambda$ and with $x(0) = x_0$, and satisfying
 $\inf_{t\in[0,\infty)}
\intrinsicdist(x(t),y(t)) =0$.
\end{defn}

\begin{rem}\rm
 Definition \ref{LM2} is stated in the context of ESIC domains that are open subsets of Euclidean spaces.
Note however that the concepts of Definition \ref{LM2} still make sense in the more general context of
\(\CAT\kappa\) metric spaces.
\end{rem}

\begin{rem}
 \rm
In contrast to the classical formulation given in \citet{Littlewood-1986}, we consider an
evasion strategy to fail if the Lion is able to approach arbitrarily close to the Man, even if the Lion does not catch the Man in finite time.
\end{rem}


\begin{rem}
\rm
%
(i) Assuming that $F_x,F_y, x_0$ and $y_0$ are
given, the conditions $x'(t) = F_x(t,x(\,\cdot\,),
y(\,\cdot\,))$ and $y'(t) = F_y(t,x(\,\cdot\,),
y(\,\cdot\,))$ specify a system of differential equations,
typically with 
right-hand sides that are discontinuous when viewed as time-varying vector fields. We do not make
any claims in general about existence or uniqueness of solutions to this
set of equations. It is trivial to
see that, for any $x_0$ and $y_0$, there exist $(x, y,
F_x, F_y)\in \Lambda$ satisfying $x(0) = x_0$ and $y(0)
= y_0$.  For example, $x$ and $y$ can be constant functions and
$F_x\equiv F_y\equiv 0$. 

(ii) For curves $x$ and $y$ that represent the Lion and Man, we will tacitly assume that if, for some $t$, $x(t) =y(t)$, then $x(s) = y(s)$ for all $s\geq t$.
\end{rem}

We introduce notation to represent pursuit games in which the 
Lion has a ``fixed path'' strategy that does not ``take into account" the strategy of the Man.
Such strategies provide a useful heuristic to understand the difference of roles for the Lion and
the Man in pursuit problems, but will not be directly employed in any of the proofs in the paper.
\begin{defn}
The set $\Lambda_0\subset \Lambda$ is the collection of all
$(x, y, F_x, F_y)\in \Lambda$ such that
$F_x(t,x(\,\cdot\,), y(\,\cdot\,))$ does not depend
on $y(\,\cdot\,)$.
\end{defn}






\begin{lem}
Suppose that $x_0$ and $y_0$ are given. The following conditions are equivalent.

(i) There exists $F_y$, with $(x, y, F_x,
F_y)\in \Lambda$ for some $x,y$ and $F_x$, and with
$x(0) = x_0$ and $y(0) = y_0$, such that
$\inf_{t\in[0, \infty)} \idist(x(t),y(t))
>0$ for every choice of $x$, $y$ and $F_x$ satisfying $x(0) = x_0$ and $y(0) = y_0$,    
and $(x, y, F_x, F_y)\in \Lambda$.

(ii) There exists $F_y$, with $(x, y, F_x,
F_y)\in \Lambda$ for some $x,y$ and $F_x$, and  with
$x(0) = x_0$ and $y(0) = y_0$, such that
$\inf_{t\in[0, \infty)} \idist(x(t),y(t))
>0$ for every choice of $x$, $y$ and $F_x$ satisfying $x(0) = x_0$ and $y(0) = y_0$,
and $(x, y, F_x, F_y)\in \Lambda_0$.  

\end{lem}

\begin{proof}
Since $\Lambda_0 \subset \Lambda$, (i) implies (ii).
Suppose that (ii) holds and consider any fixed $(x, y,
F_x, F_y)\in \Lambda$ satisfying $x(0) = x_0$
and $y(0) = y_0$. 
Let $\wh F_x(t, x,y) = \lim_{s\uparrow 0} (x(t+s) - x(t))/s $
if the limit exists and $\wh F_x(t, x,y) = 0$ otherwise. 
Since $x'(t)$ exists for almost all $t$,
$(x, y,\wh F_x, F_y)\in \Lambda_0$. Since $\liminf_{t\to \infty}
\idist(x(t),y(t))
>0$ is true for $(x, y, \wh F_x, F_y)$, it also holds for $(x,
y, F_x, F_y)$. This implies (i).
\end{proof}

Thus, the existence of a successful evasion strategy
does not depend on whether the Lion is ``intelligent". 
This may seem counterintuitive, so we offer a heuristic
explanation. The Lion may choose his strategy randomly
and may capture the Man by pure luck. The Man has to 
protect himself against all strategies, even those chosen
randomly. 

We note that our heuristic explanation
is just that---there is no randomness in the mathematical
model discussed in the lemma.
Moreover, one should be aware of the following subtle point: $\Lambda_0$ does not necessarily rigorously
correspond to the intuitive concept of the Lion choosing his strategy without regard to
the Man's position, since $F_x$ need not uniquely determine the Lion's path (due to
possible bifurcations of $x^{\prime} = F_x$).  

On the other hand a ``fixed path'' strategy for the Man 
may fail to successfully evade the Lion if the Lion is intelligent,
that is, if the Lion can
base his strategy $F_x$ on both $x$ and $y$. 
See Remark \ref{o23.6} 
at the end of the section.

We now turn to our main result on pursuit-evasion in this section.
Theorem \ref{thm:srb} states that the existence of a stable rubber band 
makes it possible for the Man to evade the Lion, 
as long as the Man starts on the rubber band and the Lion is initially a 
positive distance away from the Man.
Intuitively, this is plausible since the Man simply has to run away from the Lion  along the rubber band. The proof involves making this observation precise.
\begin{thm}\label{thm:srb}
Suppose that $D$ is an ESIC domain 
that contains a stable rubber band $\lp$.
Then there is a successful evasion strategy for the Man whenever the
starting positions $x_0, y_0\in\ol
D$ are such that $y_0 \in \lp$ and $x_0\ne y_0$.
\end{thm}

\begin{proof}
We will show there is a \(\alpha>0\), depending on \(\idist(x(0),y(0))>0\), such that,
no matter what strategy is adopted by the Lion, the Man can choose a strategy to ensure that
\(\idist(x(t),y(t))>\alpha\) for all \(t\geq0\).

Suppose that $\lp$ is a stable rubber band and Definition \ref{o19.1}(b) is satisfied
for some $\eps >0$ and function $\delta(\eta,\eps)$. 
It is evident from Definition \ref{o19.1}(b) that $\delta(\eta,\eps)$ may be chosen to be non-decreasing in $\eps$ for $\eps\geq\eta$.

Assume that $\lp$ is a stable rubber band, $y_0 \in \lp$ and $x_0\ne y_0$.
Let $\eps>0$ and $\delta(\eta,\eps)$ be such that Definition \ref{o19.1}(b)
is satisfied. We decrease $\eps$, if necessary,  so that 
\(\idist(x_0,y_0) \geq \eps/2\).

Since $D$ is ESIC, it is also \(\CAT\kappa\) for some \(\kappa\geq0\).
To ensure the global geometry of \(D\) does not interfere, we 
decrease $\eps>0$ further, if necessary, so that \(\eps<\pi/\sqrt{\kappa}\).

The essence of the argument involves the notion of \emph{hot pursuit} --
for a fixed \(\eps>0\),
we say that \emph{the Lion \(x\) is in \(\eps\)-hot pursuit of the Man \(y\) over the time interval \([T_0,T_1]\)} if
\[
 \idist(x(t),y(t)) \quad\leq\quad \eps \qquad \text{ for } T_0\leq t\leq T_1\,.
\]
We shall show that a Man can always evade a Lion in hot pursuit by running in a judiciously chosen direction along \(\lp\).
On the other hand, the Lion gains nothing by desisting from hot pursuit for a while, since an ``up-crossing argument" applied to \(\idist(x,y)\)
shows that the Man can deal with such variations simply by taking rest-periods in the intervals 
$[s,t]$ satisfying \(\idist(x(s),y(s))\geq \eps\) and \(\idist(x(u),y(u))\geq \eps/2\) for $u\in [s,t]$.

\textbf{(i)} Consider first the situation in which the Lion \(x\) begins at location \(x(0)\), at intrinsic distance at least \(\eps/6\) from \(K\) and at most \(\eps\) from the Man, who begins at
\(y(0)\in\lp\). 
Without loss of generality, we suppose \(y(0)=\lp(0)\).
Choose \(\delta=\delta(\eps/6,\eps)\) as required in Definition \ref{o19.1}(b), and set \(\eta=\min\{\delta, \eps/2\}\).
The Man 
has a choice
between running ``clockwise'' (\(y(t)=\lp(t)\))
and ``counter-clockwise'' (\(\overline{y}(t)=\lp(-t)\)). We argue that, for at least one of these strategies, the Man can remain at least
 distance \(\eta/3\) from the Lion as long as the Lion continues in \(\eps\)-hot pursuit.

Arguing by contradiction,
suppose that the Lion can use \(\eps\)-hot pursuit to come within \(\eta/3\) of the Man, whichever of the two strategies is adopted by the Man.
Let \(t_c\) and \(\overline{t}_c\) be the two times at which this \(\eta/3\)-capture occurs. Note that it is possible for either or both of \(t_c\), \(\overline{t}_c\) to
exceed the length of the loop \(\lp\); the chase may encircle \(\lp\) several times.

Define non-negative integers \(n\) and \(\overline{n}\) by
\begin{align*}
 (n-1) \ell_\lp            \quad<\quad & t_c            \quad\leq\quad n \ell_K\,,\\
 (\overline{n}-1) \ell_\lp \quad<\quad & \overline{t}_c \quad\leq\quad \overline{n} \ell_K\,,
\end{align*}
and determine rectifiable paths by using the two \(\eps\)-hot pursuits $x$ and $\overline{x}$ of the Lion:
\begin{align*}
 \Gamma_1            \quad&=\quad    x |_{[0,t_c]} \,,         \\
 \overline{\Gamma}_1 \quad&=\quad    \overline{x} |_{[0,\overline{t}_c]} \,,         \\
 \Gamma_2            \quad&=\quad    \text{minimal geodesic from } x(t_c) \text{ to } \lp(t_c) \,,         \\
 \overline{\Gamma}_2 \quad&=\quad    \text{minimal geodesic from } \overline{x}(\overline{t}_c) \text{ to } \lp(-\overline{t}_c) \,,         \\
 \Gamma_3            \quad&=\quad    \text{arc of } \lp \text{ from } \lp(t_c) \text{ to } \lp(n \ell_K)=\lp(0) \,,          \\
 \overline{\Gamma}_3 \quad&=\quad    \text{arc of } \lp \text{ from } \lp(-\overline{t}_c) \text{ to } \lp(-\overline{n} \ell_K)=\lp(0) \,.        
\end{align*}
Here, \(\Gamma_3\), \(\overline{\Gamma}_3\) are defined by continuing the same direction of travel along \(\lp\) as given by \(y\), \(\overline{y}\) respectively.
The construction of \(\Gamma_1\), \(\Gamma_2\), and \(\Gamma_3\) is illustrated in Figure \ref{fig:stable-rb1}.
\begin{figure}[thbp]
\centering
\includegraphicsKB[width=3in]{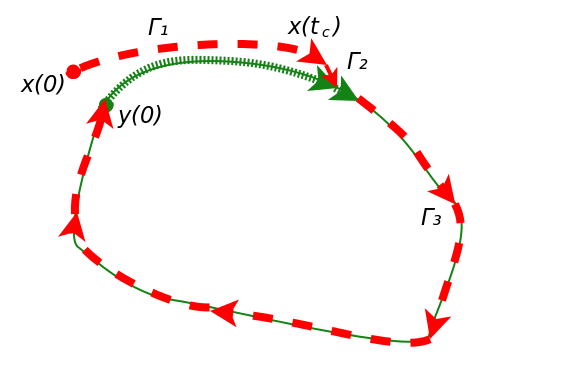}
\caption{One half of the ``hot-pursuit'' loop $\lp'$ from part (i).
}
\label{fig:stable-rb1}
\end{figure}

One can extend the pursuit by the Lion of the Man past time \(t_c\), respectively time \(\ol t_c\), depending on the strategy adopted by the Man,
along the concatenated paths \(\Gamma_1*\Gamma_2*\Gamma_3\), respectively \(\overline{\Gamma}_1*\overline{\Gamma}_2*\overline{\Gamma}_3\),
with both the Lion and the Man moving at unit speed along the extensions. 
Since the Lion is within distance \(\eta/3\le \eps\) of the Man at time \(t_c\), respectively \(\ol t_c\), 
it follows that the paths \(\Gamma_1*\Gamma_2*\Gamma_3\) and \(\overline{\Gamma}_1*\overline{\Gamma}_2*\overline{\Gamma}_3\) are both
\(\eps\)-hot pursuits of \(y\), \(\overline{y}\). We will consider the rectifiable loop running from \(x(0)=\overline{x}(0)\) back to itself, given by
\[
 \lp' \quad=\quad \Gamma_1*\Gamma_2*\Gamma_3*
\overline{\Gamma}_3^{\ -1}*
\overline{\Gamma}_2^{\ -1}*\overline{\Gamma}_1^{\ -1}\,.
\]

The \(\eps\)-hot pursuit property implies that \(\lp'\) lies in \(\ol{\ball(\lp,\eps)}\) and that one can construct a mapping $H$ as in Definition \ref{o19.1} (b). 
Moreover
\(d_H(\lp,\lp')\geq \idist(x(0),\lp)\geq\eps/6\). Finally,
\begin{multline*}
 \ell_{\lp'} \quad\le\quad
t_c + \eta/3 + (n\ell_\lp - t_c)
+
\overline{t}_c + \eta/3 + (\overline{n}\ell_\lp - \overline{t}_c)\\
\quad=\quad (n+\overline{n})\ell_\lp + 2\eta/3
\quad\leq\quad (n+\overline{n})\ell_\lp + \delta\,,
\end{multline*}
which violates the stability of the rubber band \(\lp\). This 
contradiction shows that if the Lion starts from distance at least \(\eps/6\) from \(\lp\), and remains in hot pursuit
of the Man, then the Man can choose a clockwise or counterclockwise strategy so as to always remain at least distance \(\eta/3\) away from the Lion.

\textbf{(ii)} Now consider the situation in which the Lion starts at \(x(0)\) that is less than \(\eps/6\) from \(\lp\) but greater than or equal to \(\eps/2\) and less than \(\eps\) from the 
Man's starting point \(y(0)=\lp(0)\). 
Let $z$ denote the closest point to \(x(0)\) on \(\lp\), and suppose that $z=K(-u)$ for some $u>0$
(the case $u<0$ can be dealt with in an analogous way).
By the triangle inequality, \(u \geq \eps/3 \). 

Let the Man adopt the strategy \(y(t)=\lp(t)\) and consider the Lion in \(\eps\)-hot pursuit of the Man.
Suppose the Lion comes within distance \(\eta/3\) of the Man at time \(t_h\). Define 
\begin{align*}
 \Gamma_1 \quad&=\quad x|_{[0,t_h]}\,, \\
 \Gamma_2 \quad&=\quad \text{minimal geodesic from } x(t_h) \text{ to } y(t_h)\,,\\
 \Gamma_3 \quad&=\quad y|_{[t_h, m\ell_\lp - u]}\,,\\
 \Gamma_4 \quad&=\quad \text{minimal geodesic from }z=\lp(-u) \text{ to } x(0)\,,
\end{align*}
where \(m\) is the integer satisfying
\[
 (m-1)\ell_\lp - u \quad<\quad t_h \quad\leq\quad m \ell_\lp - u\,.
\]
Consider the rectifiable loop \(\lp'=\Gamma_1*\Gamma_2*\Gamma_3*\Gamma_4\) based at \(x(0)\). 
The construction of \(\Gamma_1\), \(\Gamma_2\), \(\Gamma_3\), and \(\Gamma_4\) is illustrated in Figure \ref{fig:stable-rb2}.
\begin{figure}[thbp]
\centering
\includegraphicsKB[width=3in]{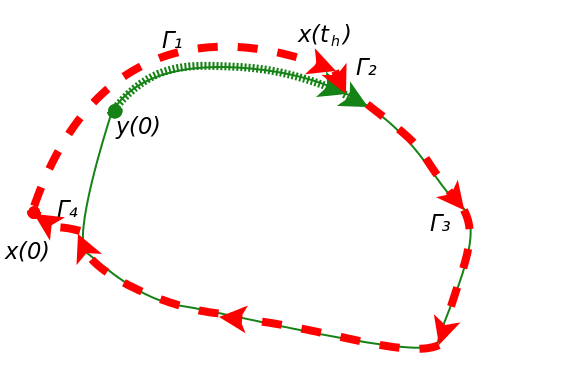}
\caption{The ``hot-pursuit'' loop from part (ii).
}
\label{fig:stable-rb2}
\end{figure}
Arguing as in (i), because of the hot pursuit by the Lion,
it follows that \(\lp'\) and \(\lp\) are homotopic within \(\ol{\ball(\lp,\eps)}\).
Once again, we choose \(\delta=\delta(\eps/6,\eps)\) as required in Definition \ref{o19.1}(b),
but we  now set \(\eta=\min\{\delta, \eps/2\}\).
 A simple computation of lengths
shows that \(\lp'\) violates the stable rubber band property of \(\lp\):
\[
 \ell_{\lp'} \quad\le\quad
t_h + \eta/3 + m\ell_\lp -u - t_h + \eps/6
\quad=\quad m \ell_\lp - (u - \eps/6 - \eta/3) \quad\leq\quad m \ell_\lp\,.
\]
So,  the Lion cannot come within  distance \(\eta/3\) of the Man if it remains in \(\eps\)-hot pursuit.

In order to complete the proof, we now spell out the up-crossing argument that was mentioned earlier. 
Suppose that $\intrinsicdist(x_0, y_0) \geq \eps >0$. 
The Man chooses to rest
until the Lion is distance \(\eps/2\) from him. We need to consider two cases.

If the Lion is within \(\eps/6\) of \(\lp\), then the Man moves in the appropriate direction given by (ii) above.
We have seen in (ii) that the Lion cannot come within  distance \(\eta/3\) of the Man while maintaining \(\eps\)-hot pursuit. 
Since neither the Lion nor the Man can travel faster than at unit speed, at least time \(\frac{1}{2}\times(\eps-\eps/2)=\eps/4\) must elapse
before the Lion and the Man are separated by \(\eps\), at which time the Man rests again.

If instead, the Lion is further than \(\eps/6\) from \(\lp\), then the Man moves in the escape direction guaranteed by (i) above. 
We have seen in (i) that the Lion cannot come within distance \(\eta/3\) of the Man while it maintains \(\eps\)-hot pursuit. Once more, 
at least time \(\frac{1}{2}\times(\eps-\eps/2)=\eps/4\) must elapse
before the Lion and the Man are separated by \(\eps\), at which time the Man rests again.

Over any finite time interval \([0,T]\), there can be at most \(1+4T/\eps\) separate periods of \(\eps\)-hot pursuit
and \(4T/\eps\) separate periods of rest and, in each of these periods, the Lion and the Man remain separated in intrinsic distance
by at least \(\min\{\eps/2,\eta/3\}=\eta/3\). Accordingly, this separation holds for all time, and therefore the Man will successfully evade the Lion.
\end{proof}

A crucial part of the above proof was the decomposition of an arbitrary pursuit strategy into alternating periods of \(\eps\)-hot pursuit
and of pursuit at a further distance. There is a related but more stringent notion of 
\emph{simple pursuit}, which will be employed in the next section.
The following definition
is adapted from \citet[Section 6]{AlexanderBishopGhrist-2009}.
\begin{defn}\label{o23.3}
We
will call $(x,y)$ a \emph{simple pursuit} if $x$
and $y$ are admissible curves, there exists a unique intrinsic geodesic 
between $x(t)$ and $y(t)$ for all $t\geq 0$, and $x'(t)$ exists for
almost all $t\geq 0$, with $|x'(t)| = 1$ for $x(t)\neq y(t)$  
and $x'(t)$ pointing towards $y(t)$ along the geodesic between $x(t)$ and $y(t)$. When $x(t_0)=y(t_0)$ for some $t_0$, 
we assume that $x(t)=y(t)$ for $t\ge t_0$.  We will
write $\Lambda_s$ to denote the family of all simple pursuits
$(x,y)$.
\end{defn}

\begin{rem}\label{o23.6.i}
\rm
Note Definition \ref{o23.3} does not assert that, for 
every $v,z \in \ol D$, there exists a unique geodesic between $v$ and $z$. This may 
not
be true for distant pairs of points.
\end{rem}

\begin{rem}
\rm
Simple pursuit can be viewed as a greedy solution to the pursuit problem (in the language of algorithm theory); it is therefore 
described as the ``greedy pursuit strategy'' in \citet{BramsonBurdzyKendall-2011}.
\end{rem}

The notion of simple pursuit is easy to illustrate in the presence of stable rubber bands, albeit in a rather elementary way.
Recall that a stable rubber band is a local geodesic (Lemma \ref{lem:locally-geodesic}).
As a consequence,
if a domain $D$ has a stable rubber band, then there exists
a simple pursuit $(x,y)$ in which the Man evades the
Lion. 
Namely, choose any $x_0,y_0\in \lp$ with
$\idist(x_0,y_0) = \eps/2$ and let $y(t)$ move
away from $x(t)$
along $\lp$ at the constant speed 1, starting from $y_0$.
Let $x(t)$ follow $y(t)$ along $\lp$ at the constant
speed 1 as well. The distance between $x(t)$ and $y(t)$
will always be $\eps/2$. Theorem \ref{thm:srb} establishes a
considerably stronger version of this fact, not limited to simple pursuit.

\begin{rem}\label{o23.6}
\rm
For some starting positions of the Lion and the Man, and some (possibly  foolish) strategies
$y(t)$ of the Man, it is evident that the Man will not evade the Lion under simple pursuit, whatever the geometry of the domain. 
For example, if $y(0)$ is in the interior of $D$
and the Man adopts the ``resting'' strategy $y(t) \equiv y_0$,
then simple pursuit from any starting point $x(0)$ close enough to 
$y(0)$ leads to capture of the Man in finite time.
\end{rem}

\section{Simple pursuit in CL domains}\label{sec:contr}


The main result of this section is Theorem \ref{thmo13}, which
shows that, if the Lion adopts an appropriate pursuit strategy, then the 
Man cannot successfully evade the Lion in
a CL domain.  We recall that since, by definition, a CL domain $D$ is ESIC, it
is also \(\CAT{\kappa}\) for some $\kappa > 0$.
Also, since the scaling $D \to \sqrt{\kappa} D$ transforms a $\CAT{\kappa}$ 
domain, $\kappa >0$, into a $\CAT{1}$ domain, it suffices to state our arguments for 
CL domains that are $\CAT1$.
(The \(\CAT0\) case is covered by the results of \citet{AlexanderBishopGhrist-2006}, where the notion 
of CL domains is not required; also note that $\CAT0$ domains are automatically $\CAT{\kappa}$, for any $\kappa >0$.)


Our strategy will be to show that if the Lion and the Man are initially close, specifically,
$\idist(x(0) , y(0)) < \pi$, then there is no successful evasion strategy by the Man
 if the Lion adopts simple pursuit (in the sense of Definition \ref{LM2}).   In particular,
we will show that, for every admissible
curve $y(\,\cdot\,)$, there exists $x(\,\cdot\,)$ such
that $(x,y)$ is a simple pursuit and $\lim_{t\to
\infty} \idist(x(t) , y(t)) =0$.  The general case, with arbitrary $x(0)$ and $y(0)$, will follow
quickly from this by constructing a chain of points in $D$ from $x(0)$ to $y(0)$, each of which
is less than distance $\pi$ from its neighbors, and applying simple pursuit at each step.

We begin by introducing a number of geometrical results that will be required in order to establish Theorem \ref{thmo13}.
To start with, we require the following general proposition from \citet[Theorem 20]{AlexanderBishopGhrist-2009}.
\begin{prop}\label{o26.40}
Suppose that $\ol D$ is a closed $\CAT{1}$ space, $x_0,y_0\in
\ol D$, $\idist(x_0,y_0)< \pi$ and $\{y(t), t\geq
0\}$ is an admissible curve. Then there exists a unique
admissible curve $\{x(t), t\geq 0\}$ such that
$(x,y)$ is a simple pursuit.
\end{prop}

 We also need
the following general geometric observation from
\citet[Proposition 23]{AlexanderBishopGhrist-2009}.
\begin{prop}\label{prop:ABG}
 Let \(\ol D\) be a compact \(\CAT1\) space. If there is a successful evasion strategy for the Man whenever
the Man is initially separated from the Lion by a distance
of less than \(\pi\), then there exists a bilaterally infinite local geodesic in $\ol D$.
\end{prop}
The idea of the proof is as follows. Consider a successful evasion by the Man of the simple pursuit strategy provided by Proposition \ref{o26.40}. 
The corresponding path will have total curvature that grows sublinearly. 
A sequence of segments of this evasion path, with lengths tending to \(\infty\), 
can be used to construct a limit by applying compactness to choose a convergent subsequence.   This limit will be
a bilaterally infinite path that must have zero total curvature, and hence be a geodesic.
(Note that this geodesic may be c1osed!)

A key result 
in this area of metric geometry
is the powerful technique of Reshetnyak majorization, which reduces the essence of many problems to calculations from two-dimensional spherical geometry.
\begin{prop}[Reshetnyak majorization]\label{prop:Reshetnyak-majorization}
 If the length of a rectifiable closed curve \(h\) in a \(\CAT1\) space \(D\) is less than \(2\pi\), then there is a 
convex domain \(C\), contained in \(S^2\), that majorizes \(h\) in the sense that there is a distance non-expanding
 map from \(C\) into \(D\), such that its restriction to the boundary of \(C\) is an arc-length preserving map onto the image of \(h\).
\end{prop}
For a proof see \citet{Reshetnyak-1968}; a clear statement can be found in \citet{ManeesawarngLenbury-2003}.

The relevant calculation from two-dimensional spherical geometry is summarized in the following preparatory lemma.
\begin{lemma}\label{lem:spherical-geometry}
 Let \(p\,q\,\widetilde{q}\,\widetilde{p}\) be a geodesic quadrilateral on the unit \(2\)-sphere, such that the interior angles
at \(p\) and \(q\) are obtuse or right-angles, 
such that \(\wt p\) and \(\wt q\) are on the same side of the great circle passing through  \(p\) and \(q\),
such that \(\delta=\dist(p,q)\in(0,\pi)\), and such that the
distances \(\dist(p,\widetilde{p})\) and \(\dist(q,\widetilde{q})\) are both bounded above by some 
positive \(\varepsilon<\delta/2\). If \(\varepsilon\)
is chosen small enough so that 
\begin{align}\label{jan9.1}
 \frac{1-\cos\delta}{\min\{\sin\delta, \sin(\delta-2\varepsilon)\}}\sin\varepsilon \quad<\quad 2,
\end{align}
then
\begin{equation}\label{eq:spherical-geometry}
\widetilde{\delta} \quad=\quad \dist(\widetilde{p}, \widetilde{q}) \quad\geq\quad
\delta - \frac{1-\cos\delta}{\min\{\sin\delta, \sin(\delta-2\varepsilon)\}} \; \sin^2\varepsilon\,.
\end{equation}
\end{lemma}
In the context of our application of this lemma we will require \(\delta\) to be small, so that we may take \(\min\{\sin\delta, \sin(\delta-2\varepsilon)\}=\sin(\delta-2\varepsilon)\).
\begin{proof}
We begin by showing how to reduce the argument to the symmetric case, 
where \(\dist(p,\widetilde{p})=\dist(q,\widetilde{q})=\varepsilon\) and the interior angles at \(p\) and \(q\) are right-angles.
First, let \(\mathcal{E}\) be the ``equatorial'' great-circle geodesic that is the perpendicular bisector of the minimal
geodesic from \(p\) to \(q\). 
Since the
distances \(\dist(p,\widetilde{p})\) and \(\dist(q,\widetilde{q})\) are bounded by \(\varepsilon<\delta/2 < \pi/2\) and the interior angles
at \(p\) and \(q\) are obtuse or right-angles, the points \(\wt p\) and \(\wt q\) lie on the opposite sides of \(\mathcal{E}\), and therefore
\[
 \dist(\widetilde{p},\mathcal{E}) + \dist(\widetilde{q},\mathcal{E}) 
\quad\leq\quad
\dist(\widetilde{p},\widetilde{q})\,.
\]
Let \(\mathcal{H}\) be the open hemisphere of \(S^2\setminus\mathcal{E}\) containing \(p\). Then the function
\(x\mapsto \dist(x,\mathcal{E})\) of \(x\in\mathcal{H}\) is a nonlinear, but strictly increasing function of the vertical height of \(x\) 
above the equatorial plane that is defined by \(\mathcal{E}\). Moreover \(x\mapsto \dist(x,\mathcal{E})\), restricted to the little
circle \(\{x: \dist(x,p)=\dist(\widetilde{p},p)\}\), can have
just one minimum and just one maximum
(since \(\dist(p,\mathcal{E})<\pi/2\)).  The maximum and minimum 
must lie on the great-circle geodesic \(\gamma\) defined by \(p\) and \(q\),
and the vertical height function \(x\mapsto \dist(x,\mathcal{E})\) varies strictly monotonically on the two connected components of \(\{x: \dist(x,p)=\dist(\widetilde{p},p)\}\setminus \gamma\).
All these facts follow immediately from the observation that the little circle \(\{x: \dist(x,p)=\dist(\widetilde{p},p)\}\)
can be obtained as the intersection of \(\mathcal{H}\) with an inclined plane.
It follows directly that
\(\dist(\widetilde{p},\mathcal{E})\) is minimized when the interior angle at \(p\) is reduced to a right-angle. Similarly, \(\dist(\widetilde{q},\mathcal{E})\) is minimized when the interior angle at \(q\) is reduced to a right-angle.

In the case where the interior angle at \(p\) (respectively \(q\)) is a right angle,
we can argue that \(\dist(\widetilde{p},\mathcal{E})\) (respectively, \(\dist(\widetilde{q},\mathcal{E})\)) is minimized
when \(\dist(\widetilde{p},p)\) (respectively, \(\dist(\widetilde{q},q)\)) is increased to 
the maximum allowed value, namely \(\varepsilon\). 
For a similar argument shows that the height function \(x\mapsto \dist(x,\mathcal{E})\), when restricted to the great circle through \(p\) and perpendicular to \(pq\) at \(p\), attains its maximum at \(x=p\), and is strictly increasing on the two portions of this geodesic rising from \(\mathcal{E}\) to \(p\).

On the other hand, if the interior angles at \(p\) and \(q\) are right-angles, and \(\dist(\widetilde{p},\mathcal{E})=\dist(\widetilde{q},\mathcal{E})=\varepsilon\), then the geodesic segments realizing 
\(\dist(\widetilde{p},\mathcal{E})\) and \(\dist(\widetilde{q},\mathcal{E})\)
will together form the minimal geodesic from \(\widetilde{p}\) to \(\widetilde{q}\). This highly symmetric situation can 
be analyzed using vector geometry. It is immediate 
from the reduction argument that \(\widetilde{\delta}=\dist(\widetilde{p},\widetilde{q})<\delta=\dist(p,q)\).  
So, we can employ Cartesian coordinates such that:
\begin{itemize}
 \item[] \((1,0,0)\) is the point of intersection of \(\mathcal{E}\) with the minimal geodesic from \(p\) to \(q\);
\item[] \(p=(\cos(\delta/2),\sin(\delta/2),0)\) and \(q=(\cos(\delta/2),-\sin(\delta/2),0)\);
\item[] \(\widetilde{p}=(\cos(\delta/2)\cos\varepsilon,\sin(\delta/2)\cos\varepsilon,\sin\varepsilon)\)
and \(\widetilde{q}=(\cos(\delta/2)\cos\varepsilon,-\sin(\delta/2)\cos\varepsilon,\sin\varepsilon)\).
\end{itemize}
Accordingly, 
\begin{equation}\label{eq:spherical-side-calculation1}
 \cos\widetilde{\delta}\;=\; 
\cos\dist(\widetilde{p},\widetilde{q})\;=\;
\cos^2(\delta/2)\cos^2\varepsilon-\sin^2(\delta/2)\cos^2\varepsilon+\sin^2\varepsilon 
\;=\;
\cos\delta \cos^2\varepsilon+\sin^2\varepsilon\,.
\end{equation}
Set \(\eta=\delta-\widetilde{\delta}\) and note that \(\widetilde{\delta}\geq\delta-2\varepsilon\) by the triangle inequality, and so
\(\eta\leq 2\varepsilon\). Note also that we assumed \(\varepsilon<\delta/2=\tfrac{1}{2}\dist(p,q)\), and so \(\widetilde{\delta}>0\).
Re-arranging \eqref{eq:spherical-side-calculation1} to read
\begin{equation}\label{eq:spherical-side-calculation2}
 \cos(\delta-\eta) - \cos\delta \quad=\quad (1-\cos\delta)\sin^2\varepsilon\,,
\end{equation}
and using \(2\varepsilon<\delta<\pi\) and \(0<\eta\leq 2\eps\), the left-hand side of \eqref{eq:spherical-side-calculation2} has 
partial derivative with respect to \(\eta\) given by 
\[
 \sin(\delta-\eta) \quad\geq\quad \min\{\sin\delta, \sin(\delta-2\varepsilon)\} \quad>\quad0\,.
\]

Since \(\pi>\delta>\delta-2\varepsilon>0\), we deduce from \eqref{jan9.1} that 
\begin{multline*}
(1-\cos\delta)\sin^2\varepsilon \quad<\quad 2 \min\{\sin\delta, \sin(\delta-2\varepsilon)\}\sin\eps \quad<\quad 
\\
\quad<\quad
\min\{\sin\delta, \sin(\delta-2\varepsilon)\}\cdot 2\eps    
\quad\leq\quad 
\int_0^{2\varepsilon}\sin(\delta-\eta)\d \eta\,.
\end{multline*}
Calculus therefore shows that \eqref{eq:spherical-side-calculation2} must have a root \(\eta\) in the range \([0,2\varepsilon)\).
Moreover, \(\eta\) must satisfy
\[
\eta \cdot \min\{\sin\delta, \sin(\delta-2\varepsilon)\}
\quad\leq\quad
 (1-\cos\delta)\sin^2\varepsilon\,,
\]
and hence
\[
 \eta 
\quad\leq\quad
 \frac{1-\cos\delta}{\min\{\sin\delta, \sin(\delta-2\varepsilon)\}}\sin^2\varepsilon,
\]
which implies \eqref{eq:spherical-geometry} as required.
\end{proof}

Let a \emph{\(k\)-times broken geodesic} 
be a continuous path which is locally geodesic save at \(k\) distinct points.
Consider now the bilaterally infinite geodesic guaranteed by Proposition \ref{prop:ABG} under a successful evasion strategy.
Given any \(T>0\), we can choose a segment of this geodesic of length at least \(T\).  By concatenating it with the
reverse curve (i.e., the geodesic segment obtained by retracing the path of the original segment),  
one obtains a closed, twice-broken
geodesic of length at least $2T$.   
 The \(\CAT1\) property constrains the constant 
of contractibility for broken geodesics as follows.

\begin{prop}\label{prop:k-broken}
Let \(D\) be an ESIC domain
that is  \(\CAT1\).
Let \(\lp\) be a loop in \(D\) that is a \(k\)-times broken geodesic. 
Suppose that $\lp$ is well-contractible, with contractibility constant $c$. Then
\begin{equation}
\label{4.5'}
 c \quad\leq\quad \frac{7+11k}{\ell_\lp}\,.
\end{equation}
\end{prop}

\begin{proof}
It suffices to show the following: for all sufficiently small $\eps>0$, any loop \(\wt K\) with
\begin{equation}\label{eq:closeness}
 \sup_t\{\idist(\wt K(t),K(t))\}\quad<\quad \eps
\end{equation}
must satisfy
\begin{equation}\label{eq:contractibility-estimate}
\ell_\lp - \ell_{\wt \lp} \quad\leq\quad (7+11k)\;\eps + 9 \ell_\lp \eps^2\,.
\end{equation}
Fix some \(\eps>0\) with \(\eps<\pi/26\). 
Partition \(\lp\)
into a sequence of \(n=\lfloor \ell_\lp/(7\eps)\rfloor\) 
segments so that
the first \(n-1\) of these segments 
have length exactly \(7\eps\), and so that the \(n^\text{th}\) segment has length in the range \([7\eps, 14\eps)\)
and is located so that it contains a ``broken point" \(p\) of the geodesic that is at distance
at least \(3\eps\) from each end-point of the segment.

Consider any loop \(\widetilde{\lp}\) satisfying \eqref{eq:closeness}. For each end point \(q'\) of 
the geodesic segments used to partition \(\lp\), project \(q'\) to the nearest
point \(\widetilde{q}\) of \(\widetilde{\lp}\), and then project \(\widetilde{q}\)
back to the nearest point \(q\) of \(\lp\). Using the triangle inequality,
it follows that these points \(q\) divide \(\lp\) into
a sequence of new segments of lengths \(\ell_1\), \(\ell_2\), \ldots, \(\ell_n\), such that
\[
 3\eps \quad<\quad \ell_i \quad<\quad 11\eps\qquad \text{ for } i=1, \ldots, n-1\,,
\]
and \(3\eps< \ell_n<18\eps\). We associate with each endpoint \(q\) of these new segments the point \(\widetilde{q}\) on \(\widetilde{\lp}\) that was chosen as above, with the points arranged in corresponding order along \(\widetilde{\lp}\).
The construction of \(q'\), \(\widetilde{q}\) and \(q\) is illustrated in Figure \ref{fig:contractible}.
\begin{figure}[thbp]
\centering
\includegraphicsKB[width=3in]{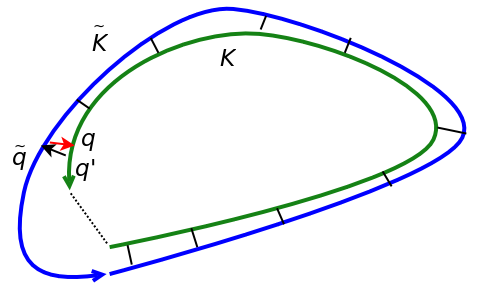}
\caption{A twice-broken geodesic \(\lp\), together with an approximating loop \(\widetilde{\lp}\). 
The twice-broken geodesic is divided into a sequence of segments as described in the text.
}
\label{fig:contractible}
\end{figure}

At most \(k\) of these segments, including the \(n^\text{th}\) segment, contain broken points of the geodesic in their interior.
For a given segment \(i=1, \ldots, n-1\), where the segment  
is \emph{not} a broken geodesic, 
this segment, together with the two points on \(\widetilde{\lp}\) corresponding to the segment's end-points, 
defines a quadrilateral in \(D\). One side of the quadrilateral is a geodesic of length \(\ell_i\), 
the two neighboring sides are also geodesics and form obtuse angles
or right-angles with the first side
(because the points \(q_i\) are projections on \(\lp\) of points in \(\widetilde\lp\)).
The fourth side can be replaced by a shorter geodesic of length \(\widetilde{\ell}_i\). 
Using the triangle inequality, the total perimeter length of the quadrilateral corresponding to the segment $i$, with
\(i<n\), will therefore be at most \(26\eps<\pi\). 
(Note that the curve formed by the quadrilateral may intersect itself, but will be a rectifiable closed curve.)

Each such quadrilateral is majorized by a convex domain $C$ in $S^2$, in the sense of Proposition \ref{prop:Reshetnyak-majorization}, with the map from $C$ to the quadrilateral being distance non-expanding
and the restriction to the boundary being arc-length preserving.  The domain $C$ is therefore a geodesic quadrilateral
that is \emph{not} self-intersecting;
its angles corresponding to the obtuse angles or right-angles of the quadrilateral in $D$ must themselves be obtuse or right-angled.  
One can also check that 
the quadrilateral $C$ satisfies the other
assumptions in Lemma \ref{lem:spherical-geometry}.  
It therefore follows from the lemma that
\[
 \widetilde{\ell}_i \quad\geq\quad \ell_i  
- \frac{1-\cos\ell_i}{\min\{\sin\ell_i, \sin(\ell_i-2\varepsilon)\}} \; \sin^2\eps
\,.
\]
Using \(3\eps < \ell_i<11\eps<\pi/2\), it also follows that
\[
  \widetilde{\ell}_i \quad\geq\quad 
\ell_i
-
\frac{1-\cos(11\eps)}{\sin\eps}\sin^2\eps
\quad\geq\quad
\ell_i  
- \frac{121 \eps^2}{2\sin\eps} \sin^2\eps
\quad\geq\quad\ell_i - 61 \eps^3
\,.
\]
This allows us to generate a lower bound on the total length of \(\widetilde{\lp}\). 
Allowing for the \(k\) or fewer segments that contain broken points of the geodesic,
this length has lower bound
\[
\ell_{\wt\lp}\quad\geq\quad
\sum_i \ell_i - (18 + 11(k-1))\eps
- 61 \left(\frac{\ell_\lp}{7\eps} - k\right) \eps^3
\quad\geq\quad
\sum_i \ell_i - (7+11k)\eps - 9 \ell_\lp \eps^2\,.
\]
Since \(\sum_i \ell_i=\ell_\lp\), we can rearrange terms to obtain \eqref{eq:contractibility-estimate}.
Considering arbitrarily small $\eps>0$,
and comparing to the definition of well-contractibility in Definition \ref{def:rubband}(b),
we obtain the upper bound on the contractibility constant \(c\)
given by (\ref{4.5'}).
\end{proof}


Propositions \ref{prop:ABG} and \ref{prop:k-broken} are the main ingredients in the proof of Theorem \ref{thmo13}.
For the proof, we also note that, by the first variation formula for $\CAT1$ spaces, the intrinsic distance 
between the Lion and the Man is non-increasing for simple pursuit (see \cite[A.1]{AlexanderBishopGhrist-2009}).

\begin{thm}
\label{thmo13}
Suppose that \(D\) is a bounded CL domain. 
Then there is no successful evasion strategy for the Man
if the Lion and Man are initially closer than \(\pi\) and 
if the Lion conducts a simple pursuit.
Moreover, there is a pursuit strategy for the Lion for which there is no successful evasion strategy for the Man, 
irrespective of the initial positions of the Lion and the Man.
\end{thm}

\begin{proof}
The CL domain \(D\) is ESIC, and is
 hence  \(\CAT\kappa\) for some \(\kappa>0\).
Rescaling if necessary, we may suppose that \(D\) is \(\CAT1\).

We first assume that the Lion and the Man are initially closer than \(\pi\), and afterwards consider the general case.
On account of Proposition \ref{prop:ABG}, a successful evasion strategy by the Man in response to
the Lion's simple pursuit would result in the construction of a bilaterally infinite local geodesic.  By the comment
preceding Proposition \ref{prop:k-broken}, one obtains closed, arbitrarily long
twice-broken geodesics in \(D\).
But, by Proposition \ref{prop:k-broken}, a closed twice-broken geodesic of length \(T\) will have contractibility constant
\(c\leq\tfrac{29}{T}\to0\) as \(T\to\infty\), which violates the assumption that \(D\) is a CL domain.
Consequently, there can be no successful evasion strategy of the Man. In particular, for any $\eps > 0$, the Lion will,
at large enough times, remain within this distance of the Man.

In order to extend the above argument to arbitrary initial positions of the Lion and the Man, one can connect these
positions by a finite chain of points that are each within distance $\pi$ of their immediate neighbors.  The 
above argument for simple pursuit by the Lion of the Man can be 
applied to each pair of neighboring points. Therefore, in each case, the distance 
between the corresponding pairs of paths will,
for large enough times, be within distance  $\eps$ of one another.  The distance between the Lion and the Man,
starting from arbitrary initial positions, will therefore eventually be within $n\eps$ of one another, 
where $n$ is the length of the chain.  Since
$\eps > 0$ is arbitrary, this completes the proof.
\end{proof}

\section{Shy couplings}\label{sec:shy}

The main result in this section is Theorem \ref{thmo24}, where we show that
CL domains admit no shy couplings.  
To demonstrate Theorem \ref{thmo24}, we will relate shy couplings to the
deterministic Lion and Man problem, with the Lion adopting the simple pursuit strategy to pursue the Man.  We 
employ a limiting Cameron-Martin-Girsanov transformation to make this comparison, after which we apply the first 
part of Theorem \ref{thmo13} to the corresponding  Lion and Man problem.


Let $D$ is a bounded ESIC domain, which is therefore CAT($\kappa$) for some $\kappa > 0$. After rescaling, we can
set $\kappa =1$, and so can assume that $D$ is CAT(1).  For any $v,z\in \ol D$ 
with $0<\idist(v,z) < \pi$, there is a unique geodesic 
between them; we denote by $\chi(v,z)$ the unit tangent vector at $v$ of the geodesic from $v$ to $z$.
(This tangent vector gives the direction of pursuit by the Lion of the Man when the Lion adopts simple pursuit.)
Proposition \ref{o26.12} states that \(\chi(v,z)\) varies continuously in \((v,z)\).
It is proved in \citet[Proposition 12 part (3)]{BramsonBurdzyKendall-2011}.

\begin{prop}\label{o26.12}
Suppose that $D$ is a bounded ESIC domain that is \(\CAT1\) .
Then
the vector field \(\chi(v,z)\) varies continuously in \((v,z)\) on the set 
$\{(v,z):\idist(v,z) \in(0, \pi)\}$, and hence is uniformly
continuous on any compact subset of this region.
\end{prop}

Proposition \ref{o26.12} allows us to prove the following useful result  about simple pursuit in CAT(1) domains.
It states that if, for each pair of initial values $x(0)$ and $y(0)$, the
paths \(x\) and \(y\) 
eventually become arbitrarily close at certain times,
then this occurs uniformly, not depending on $x(0)$ and $y(0)$. 
Recall that the family $\Lambda_s$ of simple pursuits is defined in Definition \ref{o23.3}.  

\begin{prop}\label{o28.1}
Suppose that $D$ is a bounded ESIC domain that is \(\CAT1\),
and suppose that, for each $(x,y)\in \Lambda _s$ 
with $\idist(x(0),y(0)) \leq \pi/2$, and each $\eps>0$, there exists some 
$t< \infty$ such that $\idist(x(t), y(t)) \leq \eps$. 
Then, for each $\eps>0$, there exists some $t<\infty$ such
that, for each $(x,y)\in \Lambda _s$ with $\idist(x(0),y(0)) \leq \pi/2$, 
$\idist(x(t), y(t)) \leq \eps$.
\end{prop}
\begin{proof}
Assume that, on the contrary, there exists $\eps>0$ such that, for every  integer $n>0$, there exists 
$(x_n,y_n)\in \Lambda _s$ such that $\idist(x_n(0),y_n(0)) \leq \pi/2$
and $\idist(x_n(n), y_n(n)) \geq \eps$.  (As observed by \citealp[A.1]{AlexanderBishopGhrist-2009}, 
the function $t\to\idist(x(t), y(t))$ 
is non-increasing for simple pursuit; we consequently need only consider integer times $n$.)

Since \(D\) is ESIC, we may extend simple pursuit to \(\ol D\) as well.
This allows us to use a variation on the classic Arzela-Ascoli argument.
Using the compactness of $\ol D$ and passing to a subsequence if necessary, 
we can assume that $\{x_n(0)\}_{n\geq 1}$ 
converges to $x_\infty(0)$ and $\{y_n(0)\}_{n\geq 1}$ 
converges to $y_\infty(0)$;
one must have $\idist(x_\infty(0), y_\infty(0))\leq \pi/2$. 
The functions $x_n$ and $y_n$ are Lipschitz with constant 1
so, for every fixed interval $[0,k]$, there exist subsequences 
of $\{x_n\}$ and $\{y_n\}$ that converge to admissible functions 
$x_\infty$ and $y_\infty$ uniformly on $[0,k]$. Using the diagonal method, 
we can assume that $\{x_n\}$ and $\{y_n\}$ converge 
to admissible functions $x_\infty$ and $y_\infty$ uniformly on every compact interval.
This and $\inf_{0\leq t \leq n}\idist(x_n(t), y_n(t)) \geq \eps$ 
imply that, for every 
$n$, $\inf_{0\leq t \leq n}\idist(x_\infty(t), y_\infty(t)) \geq \eps$. 
Hence, $\inf_{0\leq t <\infty}\idist(x_\infty(t), y_\infty(t)) \geq \eps$. 

It follows from  the uniform convergence of $x_n$ to 
$x_\infty$ and Proposition \ref{o26.12} that
$ \chi(x_n(s), y_n(s)) \to \chi(x_\infty(s), y_\infty(s))$ for every $s$. 
Since $x_n(t) - x_n(0) = \int_0^t \chi(x_n(s), y_n(s)) ds$ for 
every $t\geq 0$,
by applying bounded convergence as $n\rightarrow\infty$, it follows that 
$x_\infty(t) - x_\infty(0) = \int_0^t \chi(x_\infty(s), y_\infty(s)) ds$ 
for every $t\geq 0$. This shows that $(x_\infty, y_\infty)\in \Lambda_s$. 
Therefore, by the assumption made in the proposition, 
$\idist(x_\infty(t), y_\infty(t)) \to 0$ as $t\to\infty$. 
This contradicts our earlier claim and hence completes the proof. 
\end{proof}


We next introduce the notion of coupled Brownian motions. As mentioned in the introduction,
all probabilistic couplings considered in this paper are assumed to be co-adapted -- Brownian 
motions \(B\) and \(\widetilde{B}\) are
\emph{co-adaptively coupled} if they are
defined on the same probability space, are adapted to the same filtration
\(\{\mathcal{F}_t:t\geq0\}\) and if, in addition, both have independent
increments \emph{with respect to their common filtration}, i.e.,
\begin{align*}
 B_{t+s}-B_t &\text{ is independent of }\mathcal{F}_t\text{ for all }t, s\geq0\,, \quad\text{and}
\\
 \widetilde{B}_{t+s}-\widetilde{B}_t &\text{ is independent of }\mathcal{F}_t\text{ for all }t, s\geq0.
\end{align*}
(The alternative terminology of ``jointly immersed'' Brownian motions makes explicit use of the theory of
co-immersed filtrations of \(\sigma\)-algebras, see \citealp{Emery-2005}.)
Note that \(B_{t+s}-B_t\) and \(\widetilde{B}_{t+s}-\widetilde{B}_t\) will not in general be independent of each other.
\citet[Lemma 6]{Kendall-2009a} gives an explicit proof of the result from the folklore of stochastic calculus
that one may represent such a coupling using stochastic integrals,
possibly at the cost of augmenting the filtration so as to include a further independent Brownian 
motion \(C\).  Namely, there exist
\((d\times d)\)-matrix-valued predictable random processes \(\mathbb{J}\) and \(\mathbb{K}\) such that
\begin{align}\label{o26.1}
\wt B
\quad=\quad \int \mathbb{J}^\top \d B + \int \mathbb{K}^\top \d C\,,
\end{align}
with $\mathbb{J}$ and $\mathbb{K}$ satisfying
\begin{equation}
\label{new26.1}
\mathbb{J}^\top\mathbb{J}+\mathbb{K}^\top\mathbb{K} = I_d
\end{equation}
at all times, where $I_d$ is the \((d\times d)\) identity matrix.
(Informally one may view \(\mathbb{J}\) as the matrix of infinitesimal covariances between
the Brownian differentials \(\d B\) and \(\d \wt B\).)

In the context of stochastic calculus, a
pair of processes \(X\) and \(\widetilde{X}\) is said to form a co-adapted coupling
if they can be defined by strong solutions of stochastic differential equations driven
by \(B\), \(\widetilde{B}\) respectively. 
(There is of course a wider theory of co-adapted coupling applying to general Markov chains
and other random processes.) 
We will employ the stochastic differential
equation obtained from the Skorokhod transformation for reflected Brownian motion in an ESIC domain \(D\).
\citet{Saisho-1987} has shown 
for ESIC domains that, given
a driving Brownian motion \(B\),
there exists a unique solution pair \((X, \int \nu_X\d L^X)\) satisfying
\begin{align}
 &\d X \quad=\quad \d B - \nu_X \d L^X\,,\nonumber\\
 &L^X\text{ is non-decreasing and increases only when }X\in\partial D\,,
\label{6.2.18.new}\\
 &\nu_X \in \NN_{X,r}\,.\nonumber
\end{align}
Here, \(L^X\) may be viewed as the local time of the reflected Brownian motion \(X\) on the boundary \(\partial D\),
while \(\nu_X\) is a unit vector defined only when \(X\in\partial D\). (In the case of smooth boundary, \(\nu_X\)
may be taken to be the unit outward-pointing normal vector at \(X\in\partial D\); in
the more general case with uniform exterior sphere and interior cone conditions, the definitions of \(L^X\)
and \(\nu_X\) will be interdependent, but all choices lead to the same process \(X\).)  Note that the solutions
of \eqref{6.2.18.new} are pathwise unique, and the process $X$ is strong Markov.

Consider a co-adapted coupling of reflecting Brownian motions \(X\) and \(Y\) in
the bounded ESIC domain
\(D\subset\Reals^d\).
We can use \eqref{o26.1} to represent
this coupling as
\begin{align}
\d X \quad&=\quad \d B - \nu_X\d L^X\,,\label{o26.10}\\
\d Y \quad&=\quad \left(\mathbb{J}^\top\d B+\mathbb{K}^\top\d A\right)-\nu_Y\d
L^Y\,,
\label{o26.11}
\end{align}
where \(A\) and \(B\) are independent \(d\)-dimensional Brownian motions, and
\(\mathbb{J}\), \(\mathbb{K}\) are predictable \((d\times d)\)-matrix processes
such that \eqref{new26.1} is satisfied.
%
Here \(L^X\) and \(L^Y\) may be viewed informally as the local times of \(X\) and \(Y\)
that have accumulated 
on the boundary.
We interpret the  Brownian particle \(X\) as
the Brownian Lion or pursuer, and the other Brownian particle
\(Y\) as the Brownian Man or evader.
It will be convenient for the following work to suppose that the coupling given in \eqref{o26.10}-\eqref{o26.11}
holds only up to the time \(T^*=\inf\{t\geq0:X(t)=Y(t)\}\) (the time of ``capture''); we define the coupling
for all times \(t\geq T^*\)
by \(Y(t)=X(t)\), with \(X\) satisfying \eqref{o26.10} after time \(T^*\).


The main result of this section, Theorem \ref{thmo24}, is that
a bounded CL domain
cannot support a shy coupling. Most of the work is carried out in Proposition \ref{o26.4}, which is then applied
in the proof of the theorem. We state both the proposition and the theorem first, and then give their proofs. 
These results are related to those in
\citet{BramsonBurdzyKendall-2011}, although the proofs differ in significant details.  
Theorem 1 of \citet{BramsonBurdzyKendall-2011} only holds for ESIC domains that are $\CAT0$, whereas
Theorem \ref{thmo24} here covers the more general
CL domains. 
The latter family includes, for example, star-shaped ESIC domains, and more general ESIC domains that are
mentioned in Example \ref{jan9.3}, neither of which need be \(\CAT0\).
\begin{prop}\label{o26.4} 
Suppose that the CL domain $D$ is 
bounded in the Euclidean metric.
For any $\eps>0$, there exists a $t>0$ such that, for 
any \(X\) and \(Y\) satisfying \eqref{o26.10}-\eqref{o26.11} 
with \(X(0)\), \(Y(0)\in \overline{D}\), 
\begin{align}
\label{o26.4d}
 \Prob{\inf_{0\le s\le t}\idist(X(s),Y(s))\leq \eps\ } \quad>\quad 0\,.
\end{align}
\end{prop}

\begin{rem} \rm 
Our proof of Proposition \ref{o26.4} actually yields the following stronger result.
For any $\eps>0$ and all $0 < t_1 < t_ 2 < \infty$,
\begin{align*}
\Prob{\idist(X(s),Y(s))\leq \eps\,\text{ whenever }\,{t_1\leq s\leq t_2} } \quad>\quad 0\,.
\end{align*}
The version (\ref{o26.4d}) suffices for Theorem \ref{thmo24} 
and is needed for Theorem \ref{thm:rubber-band-not-shy}.
\end{rem} 

\begin{thm}
\label{thmo24}
Suppose that the CL domain \(D\) is bounded.
Then, there is no shy 
co-adapted coupling for reflected Brownian motion in \(D\).
\end{thm}

\begin{proof}[Proof of Proposition \ref{o26.4}]
Since $D$ is CL and is hence ESIC, it can be scaled so that it is $\CAT1$.
We first demonstrate (\ref{o26.4d}) when  $\idist(X(0),Y(0)) \leq \pi/2$, which will provide
motivation for the general case. 

The first step is to alter the stochastic dynamics of the coupled Brownian motions $X$ and $Y$ given  in
\eqref{o26.10}-\eqref{o26.11}  by adding a large drift.  The new equations are given in
\eqref{o26.20}-\eqref{o26.21}; the drift there for the $X^n$ component is given by $n$ times 
the unit tangent vector field \(\chi\) introduced before Proposition \ref{o26.12}, 
and the drift of \(Y^n\) is given by adding the corresponding large drift governed by the product of the
coupling matrix \(\mathbb{J}^\top\) with $\chi$. Setting \(T^{*,n}=\inf\{t\geq0:X^n(t)=Y^n(t)\}\),
for \(t<T^{*,n}\), one has
\begin{align}
 X^n(t) \quad&=\quad X(0)+B(t) 
+\int_0^t n\chi(X^n(s), Y^n(s))\d s - \int_0^t \nu_{X^n(s)}\d L^{X^n}_s\,,\label{o26.20}\\
 Y^n(t) \quad&=\quad Y(0) + \int_0^t \left(\mathbb{J}_s^\top\d B(s)+\mathbb{K}_s^\top\d A(s)\right)  \label{o26.21}\\
& \quad\qquad + \int_0^t n
\mathbb{J}_s^\top\chi(X^n(s),Y^n(s)) \d s 
- \int_0^t \nu_{Y^n(s)}
\d L^{Y^n}_s \nonumber\,.
\end{align}
As after \eqref{o26.10}-\eqref{o26.11}, for \(t>T^{*,n}\), we set \(Y^n(t)=X^n(t)\) and let \(X^n(t)\) evolve
as the ordinary reflected Brownian motion after \(T^{*,n}\).  (Note that $\chi(v,z)$ is not defined for \(v=z\).)
We also set $T^n=\inf\{t\geq 0: \idist(X^n(t) , Y^n(t)) \geq 3\pi/4\}$. Since $\chi(v,z)$ is not necessarily 
uniquely defined if $\idist(v,z) \geq \pi$, we will need to analyze the stopped processes 
$X^n(t \land T^n)$ and $Y^n(t \land T^n)$. 

By the Cameron-Martin-Girsanov theorem, the distributions of the solutions of
\eqref{o26.10}-\eqref{o26.11} and \eqref{o26.20}-\eqref{o26.21} are
mutually absolutely continuous on every interval $[0, T^n \land T^{*,n}\land k]$, for $k<\infty$. 
On the other hand, as we will show, after rescaling time and taking \(n\) to be
very large, the paths of $(X^n(\cdot), Y^n(\cdot))$
can be viewed as being uniformly close to those for the corresponding Lion and Man problem.  
Since \(D\) is assumed to be a CL domain,
this will allow us to apply Theorem \ref{thmo13} to establish \eqref{o26.4d}. 

We rescale time by making the substitutions
\(X^n(t)=\widetilde{X}^n(nt)\),
\(Y^n(t)=\widetilde{Y}^n(nt)\), \(B(t)=\widetilde{B}^n(nt)/\sqrt{n}\),
\(A(t)=\widetilde{A}^n(nt)/\sqrt{n}\),
\(\mathbb{J}(t)={\widetilde{\mathbb{J}}^{(n)}}(nt)\),
\(\mathbb{K}(t)={\widetilde{\mathbb{K}}^{(n)}}(nt)\).
Then \eqref{o26.20}-\eqref{o26.21} take the form
\begin{align}
\wt X^n(t) \quad&=\quad X(0)+\frac{1}{\sqrt{n}} \wt B^n(t) 
+\int_0^t \chi(\wt X^n(s),\wt Y^n(s))\d s - \int_0^t \nu_{\wt X^n(s)}\d L^{\wt X^n}_s\,,\label{o26.22}\\
\wt Y^n (t) \quad&=\quad Y(0)+ \frac{1}{\sqrt{n}}
\int_0^t \left((\widetilde{\mathbb{J}}_s^{(n)})^\top\d
\widetilde{B}^n(s)+(\widetilde{\mathbb{K}}_s^{(n)})^\top\d \widetilde{A}^n(s)\right)\label{o26.23}\\
& \qquad +
\int_0^t(\widetilde{\mathbb{J}}_s^{(n)})^\top
\chi(\widetilde{X}^n(s),
\widetilde{Y}^n(s))\d s- \int_0^t\nu_{\widetilde{Y}^n(s)}
\d L^{\widetilde{Y}^n_s}\,.
\nonumber
\end{align}
As before, for \(t>\wt T^{*,n}=\inf\{t\geq0: \wt X^n(t)=\wt Y^n(t)\}\), we 
set \(\wt Y^n(t)=\wt X^n(t)\) and let \(\wt X^n(t)\) evolve
as ordinary reflected Brownian motion after \(\wt T^{*,n}\).
Note that $\wt B^n$ and $\wt A^n$ are standard Brownian motions.
Corresponding to the previous definition of \(T^n\), we define stopping times 
$\wt T^n=\inf\{t\geq 0: \idist(\wt X^n(t) ,\wt Y^n(t)) \geq 3\pi/4\}$
and analyze the stopped processes $\wt X^n(t \land \wt T^n)$ and 
$\wt Y^n(t \land \wt T^n)$. 

Now, consider the analog of \eqref{o26.22}-\eqref{o26.23}, but without boundary:
\begin{align}
\wt U^n(t) \quad&=\quad \frac{1}{\sqrt{n}} \wt B^n(t) 
+\int_0^t \chi(\wt X^n(s),\wt Y^n(s))\d s \,,\label{o26.30}
\\
\wt V^n (t)\quad&=\quad  \frac{1}{\sqrt{n}}
\int_0^t \left((\widetilde{\mathbb{J}}_s^{(n)})^\top\d
\widetilde{B}^n(s)+(\widetilde{\mathbb{K}}_s^{(n)})^\top\d \widetilde{A}^n(s)\right) \label{o26.31}
+ \int_0^t(\widetilde{\mathbb{J}}_s^{(n)})^\top
\chi(\widetilde{X}^n(s),
\widetilde{Y}^n(s))\d s\,.
\end{align}
The criterion of \citet[\S1.4]{StroockVaradhan-1979} establishes tightness of the sextuplet
\begin{align}
\bH^n(t) = \Big(&\wt U^n(t) , \frac{1}{\sqrt{n}} \wt B^n(t) , 
\int_0^t \chi(\wt X^n(s),\wt Y^n(s))\d s \,,\label{o27.3}
\\
&\wt V^n (t),  \frac{1}{\sqrt{n}}
\int_0^t \left((\widetilde{\mathbb{J}}_s^{(n)})^\top\d
\widetilde{B}^n(s)+(\widetilde{\mathbb{K}}_s^{(n)})^\top\d \widetilde{A}^n(s)\right) ,
 \int_0^t(\widetilde{\mathbb{J}}_s^{(n)})^\top
\chi(\widetilde{X}^n(s),
\widetilde{Y}^n(s))\d s \Big)\,,\nonumber
\end{align}
since
the diffusion coefficients and the drifts are bounded by \(1\). So, there exists an appropriate
subsequence of $\bH^n$ that converges weakly (in the uniform metric) to a
limiting process $\bH^\infty$.
In a harmless abuse of notation, we re-index, denoting this subsequence by $\{\bH^n:n\geq1\}$. 
In particular, $\wt U^n(t)$ and $\wt V^n(t)$ converge 
weakly so, by
\citet[Thm. 4.1]{Saisho-1987},
 \((\wt X^n,\wt Y^n)\) converges weakly to a limiting continuous process
\((\wt X^\infty, \wt Y^\infty)\) along the same subsequence. It follows
that the octuplet
\begin{align}
\bK^n(t) = \Big(&\wt X^n(t),\wt Y^n(t), \wt U^n(t) , \frac{1}{\sqrt{n}} \wt B^n(t) , 
\int_0^t \chi(\wt X^n(s),\wt Y^n(s))\d s \,,\label{o27.4}
\\
&\wt V^n (t),  \frac{1}{\sqrt{n}}
\int_0^t \left((\widetilde{\mathbb{J}}_s^{(n)})^\top\d
\widetilde{B}^n(s)+(\widetilde{\mathbb{K}}_s^{(n)})^\top\d \widetilde{A}^n(s)\right) ,
 \int_0^t(\widetilde{\mathbb{J}}_s^{(n)})^\top
\chi(\widetilde{X}^n(s),
\widetilde{Y}^n(s))\d s \Big)\nonumber
\end{align}
is tight, and therefore converges weakly along a further subsequence. Once again,
we commit a harmless abuse of notation and re-index, denoting the weakly converging subsequence by $\{\bK^n:n\geq1\}$. We
now employ the
Skorokhod representation of weak convergence to construct the sequence of $\bK^n$ on the same probability space so that it converges
almost surely, uniformly on compact intervals.

The fourth and seventh components of $\bK^n$ are Brownian motions run at rate
\(\tfrac{1}{n}\), so they each converge to the zero
process as $n \rightarrow \infty$. 
The fifth and eighth components of $\bK^n$
are both Lip\((1)\);
their limits are therefore also  Lip\((1)\). These observations
and \eqref{o26.30}-\eqref{o26.31} imply that the limits
\(\wt V^\infty\) and \(\wt U^\infty\) of \(\wt V^n\) and \(\wt U^n\) are also \(\text{Lip}(1)\).

Let $\wt T^\infty = \liminf_{n\to \infty} \wt T^{n}$
and $\wt T^* = \inf\{t\geq 0: \wt X^\infty(t) = \wt Y^\infty(t)\}$, 
and note that \(\wt T^* \leq \liminf_{n\rightarrow\infty} \wt T^{*,n}\).
 We will argue that
\begin{align}\label{o26.60}
 \wt U^\infty(t) &= \int_0^t \chi(\wt X^\infty(s),\wt Y^\infty(s))\d s \qquad \text{  for  } t < 
\wt T^\infty \land \wt T^*, \\
 \wt T^\infty &= \infty \quad \text{ a.s.}, \label{o27.1}\\
\wt X^\infty(t) & = \wt Y^\infty(t)  
\text{  for  } t \geq  \wt T^*. \label{o27.2}
\end{align}

The bounded vector field \(\chi(\wt X^n(t),\wt Y^n(t))\) depends continuously
on \(\wt X^n(t)\) and \(\wt Y^n(t)\), over $[0, \wt T^n \land \wt T^{*,n})$, 
by Proposition~\ref{o26.12}. Hence, by 
the bounded convergence theorem, we can pass to the limit in \eqref{o26.30} on the interval
$[0, \wt T^\infty \land \wt T^*)$, with
the limit satisfying \eqref{o26.60} on $[0, \wt T^\infty \land \wt T^*)$. By Proposition
\ref{o26.40}, for given $X(0)$ and $Y(0)$ with 
$\idist(X(0),Y(0))< \pi$ and $\{\wt Y^\infty(t), t\geq 0\}$, \eqref{o26.60} 
defines a unique function 
$\{\wt U^\infty(t), t\in [0, \wt T^\infty \land \wt T^*)\} = \{\wt X^\infty(t), t\in [0, \wt T^\infty \land \wt T^*)\}$. 

Since \eqref{o26.60}  holds for $t \in[0, \wt T^\infty \land \wt T^*)$,
\(\wt X^\infty\) conducts a simple pursuit of \(\wt Y^\infty\) over this time period. 
As noted above Theorem \ref{thmo13}, 
it follows that $t\to \idist(\wt X^\infty(t),\wt Y^\infty(t))$ is non-increasing on this interval. 
Consequently, $\sup_{0 \leq t \leq \wt T^\infty \land \wt T^*} \idist(\wt X^\infty(t),\wt Y^\infty(t)) 
\leq \idist(\wt X^\infty(0),\wt Y^\infty(0)) \leq \pi/2$. Since 
$(\wt X^n,\wt Y^n)$ converges a.s. to $(\wt X^\infty,\wt Y^\infty)$ 
uniformly on compact intervals, we conclude that, for large 
$n$, $\sup_{0 \leq t \leq \wt T^\infty \land \wt T^*} \idist(\wt X^n(t),\wt Y^n(t)) \leq 5\pi/8 < 3\pi/4$. 
This contradicts the definition of $\wt T^\infty$ unless either 
\(\wt T^\infty=\infty\) or $\wt T^* < \wt T^\infty$.

Suppose that $\wt T^*<\wt T^\infty$ and $\wt X^\infty(t) \ne \wt Y^\infty(t)$ for some
$t \in (\wt T^* , \wt T^\infty)$. 
Since the processes $\wt X^\infty(t) $ and $ \wt Y^\infty(t)$
are continuous, this implies that there exist $\wt T^* < t_1 < t_2 < \wt T^\infty$ such that
$0< \idist(\wt X^\infty(t_1),\wt Y^\infty(t_1)) < \idist(\wt X^\infty(t_2),\wt Y^\infty(t_2))$
and $\inf_{t_1 \leq t \leq t_2} \idist(\wt X^\infty(t),\wt Y^\infty(t)) >0$. Therefore, 
\(t_2<\liminf T^{*,n}\).
Arguing in the same way as for \eqref{o26.60}, it follows that \(\wt X^\infty\) 
conducts a simple pursuit of \(\wt Y^\infty\) over the
interval $[t_1,t_2]$, and therefore 
$t\to \idist(\wt X^\infty(t),\wt Y^\infty(t))$ is non-increasing on this interval.
This is a contradiction, so we conclude that $\wt X^\infty(t) = \wt Y^\infty(t)$ for 
$t \in (\wt T^* , \wt T^\infty)$.
This completes the proof of \eqref{o26.60}-\eqref{o27.2} and shows that $\wt X^{\infty}$ conducts
a simple pursuit of $\wt Y^{\infty}$ over the interval $[0,\wt T^*)$.

Fix an arbitrarily small $\eps>0$. 
Since the CL domain $D$ is $\CAT1$ and $\idist(\wt X^{\infty}(0),\wt Y^{\infty}(0)) \le \pi/2$,
it follows from Theorem \ref{thmo13} that it is impossible for $\wt Y^{\infty}$ to successfully evade
$\wt X^{\infty}$ over the time interval $[0,\infty)$.  Moreover, since this holds for all such $\wt X^{\infty}(0)$ and 
$\wt Y^{\infty}(0)$, application of 
Proposition \ref{o28.1} implies that there exists $t_1 < \infty$, 
not depending on either $\wt X^{\infty}(0),\wt Y^{\infty}(0)$, $\omega$ or the particular simple pursuit in
$\Lambda_s$ of the pair ($\wt X^{\infty}, \wt Y^{\infty}$), 
such that, for $t\geq t_1$, 
\begin{align}
\label{new6.16}
\idist(\wt X^\infty(t),\wt Y^\infty(t)) \leq \eps/2.
\end{align}

Because of the uniform convergence of $(\wt X^n, \wt Y^n)$ to $(\wt X^\infty, \wt Y^\infty)$ over finite intervals, it follows from
(\ref{new6.16}) that, for some $n_0<\infty$ depending on $X(0)$ and $Y(0)$, and all $n\geq n_0$,
\begin{align*}
\Prob{
\idist(\wt X^n(t_1),\wt Y^n(t_1)) \leq \eps} \quad>\quad 0\,.
\end{align*}
Changing the clock back to the original pace, we obtain
\begin{align*}
\Prob{
 \idist( X^n(t_1/n), Y^n(t_1/n)) \leq \eps} \quad>\quad0\,.
\end{align*}
By the Cameron-Martin-Girsanov theorem,
\begin{align}\label{o28.2}
\Prob{
\idist( X(t_1/n), Y(t_1/n)) \leq \eps} \quad>\quad0\,.
\end{align}
This implies (\ref{o26.4d}) when $\idist(X(0),Y(0)) \leq \pi/2$.

We now consider (\ref{o26.4d}) for $X(0)=x_0$, $Y(0) = y_0$, and
arbitrary $x_0, y_0 \in \ol D$.  The reasoning is similiar to the case where
$\idist(X(0),Y(0)) \leq \pi/2$, after constructing a chain of points, each of which is within distance
$\pi/2$ of its immediate neighbors. 

Choose a sequence of points $z_1, z_2, \dots , z_m \in \ol D$ such that 
$z_1 = x_0$, $z_m = y_0$ and $\idist(z_k, z_{k+1}) \leq \pi/2$ for all $1\leq k \leq m-1$.
For \(n\geq1\) and \(k=1, \ldots, m\), we define the chain of random 
processes $Z^{k,n}$ in $D$, with $Z^{k,n}(0) = z_k$.  
We set \(Z^{1,n}\equiv X^n\) and \(Z^{m,n}\equiv Y^n\), but with the drift \(n\chi(X^n(s),Y^n(s))\)
replaced by \(n\chi(X^n(s),Z^{2,n}(s))=n\chi(Z^{1,n}(s),Z^{2,n}(s))\) 
for both processes.  For \(k=2, \ldots, m-1\), 
\(Z^{k,n}\) denotes the process that conducts a simple pursuit directed toward
\(Z^{k+1,n}\), but carried out at rate \(n\). The corresponding stochastic system is given by
\begin{align}
 Z^{1,n}(t) \quad&=\quad X(0)+B(t) 
+\int_0^t n\chi(Z^{1,n}(s),Z^{2,n}(s))\d s - \int_0^t \nu_{Z^{1,n}(s)}\d L^{Z^{1,n}}_s\,,\label{eq:chain-1}\\
&\ldots\nonumber\\
Z^{k,n}(t) \quad&=\quad Z^{k.n}(0) +  \int_0^t n\chi(Z^{k,n}(s),Z^{k+1,n}(s))\d s\,, \label{eq:chain-k}\\
&\ldots\nonumber\\
 Z^{m,n}(t) \quad&=\quad Y(0) + \int_0^t \left(\mathbb{J}_s^\top\d B(s)+\mathbb{K}_s^\top\d A(s)\right)  \label{eq:chain-m}\\
& \quad\qquad + \int_0^t n
\mathbb{J}_s^\top\chi(Z^{1,n}(s),Z^{2,n}(s)) \d s 
- \int_0^t \nu_{Z^{m,n}(s)}
\d L^{Z^{m,n}}_s \nonumber\,.
\end{align}
Since \(Z^{2,n}\), \ldots, \(Z^{m-1,n}\) are simple pursuits run at rate \(n\) 
and directed toward adapted processes, they are Lipschitz\((n)\) adapted random processes. 
Also, for \(k=2, \ldots, m-2\), \(\idist(Z^{k,n}(t),Z^{k+1,n}(t))\) is non-increasing in time.
(No reflection term is required in (\ref{eq:chain-k}) since $Z^{k,n}$, for $k=2,\ldots,m-1$, will
never attempt to cross the boundary.)

The system (\ref{eq:chain-1})-(\ref{eq:chain-m}) is run up until the time 
\[
S^n\;=\inf\{t\geq0:\idist(\;Z^{1,n}(t),Z^{2,n}(t))\geq 3\pi/4\} \;\land\; \inf\{t\geq0: \idist(Z^{m-1,n}(t),Z^{m,n}(t))\geq 3\pi/4\}\,.
\]
For \(k=2, \ldots, m-2\), we set \(Z^{k,n}(t)=Z^{k+1,n}(t)\) when
\(t\geq\inf\{s:Z^{k,n}(s)=Z^{k+1,n}(s)\}\). 
We adopt the convention that \(\chi(Z^{1,n}(t),Z^{2,n}(t))=0\) when \(Z^{1,n}(t)=Z^{2,n}(t)\)
and \(\chi(Z^{m-1,n}(t),Z^{m,n}(t))=0\) when \(Z^{m-1,n}(t)=Z^{m,n}(t)\).
(Almost surely, the set of times \(t\) at which either of the latter two equalities occurs 
has measure zero since, in either case, one process is a 
Brownian motion with drift and the other is a Lipschitz process.)

Rescaling time as in \eqref{o26.22}-\eqref{o26.23}, we obtain the system
\begin{align}
 \wt Z^{1,n}(t) \quad&=\quad X(0)+\frac{1}{\sqrt{n}}\wt B^n (t) 
+\int_0^t \chi(\wt Z^{1,n}(s),\wt Z^{2,n}(s))\d s - \int_0^t \nu_{\wt Z^{1,n}(s)}\d L^{\wt Z^{1,n}}_s\,,\label{eq:chain-1-retimed}\\
&\ldots\nonumber\\
 \wt Z^{k,n}(t) \quad&=\quad Z^{k.n}(0) +  \int_0^t \chi(\wt Z^{k,n}(s),\wt Z^{k+1,n}(s))\d s\,, \label{eq:chain-k-retimed}\\
&\ldots\nonumber\\
 \wt Z^{m,n}(t) \quad&=\quad Y(0) + \frac{1}{\sqrt{n}}\int_0^t \left(\wt{\mathbb{J}}_s^\top\d \wt B^n(s)+\wt{\mathbb{K}}_s^\top\d \wt A^n(s)\right)  \label{eq:chain-m-retimed}\\
& \quad\qquad + \int_0^t 
\wt{\mathbb{J}}_s^\top\chi(\wt Z^{1,n}(s),\wt Z^{2,n}(s)) \d s 
- \int_0^t \nu_{\wt Z^{m,n}(s)}
\d L^{\wt Z^{m,n}}_s \nonumber\,,
\end{align}
which is run up until time 
\[
\wt S^n\;=\;
\inf\{t\geq0:\idist(\wt Z^{1,n}(t),\wt Z^{2,n}(t))\geq 3\pi/4\} \;\land\; \inf\{t\geq0: \idist(\wt Z^{m-1,n}(t),\wt Z^{m,n}(t))\geq 3\pi/4\}\,,
\]
and which follows the conventions noted earlier when two processes coincide.

We now argue as in the case where $\idist(X(0),Y(0)) \leq \pi/2$, letting 
\(n\to\infty\) through a subsequence so that the system of solutions to \eqref{eq:chain-1-retimed}-\eqref{eq:chain-m-retimed} converges weakly to a chain of simple pursuits \(\wt Z^{1,\infty}\), \ldots, \(\wt Z^{m,\infty}\) commencing at \(z_1\), \ldots, \(z_m\).  For $k=1$, with
$\wt Z^{1,n}$ and $\wt Z^{2,n}$, the reasoning is almost
the same as before; although the process $\wt Z^{2,n}$ is different than $\wt Y^n$, in both cases their
drifts are at most $1$, and as $n\rightarrow\infty$, both result in a simple pursuit.  The steps
$k=2,\ldots,m-1$ are easier to see since, for each $n$, $\wt Z^{k,n}$ already conducts a (random)
simple pursuit of $\wt Z^{k+1,n}$.  For step $m-1$, $\wt Z^{m-1,n}$ has drift $1$ and the drift of
$\wt Z^{m,n}$ is at most $1$ and so, as $n\rightarrow\infty$, one again obtains a simple pursuit. 

As before, \(\wt S^n\to \infty\). Fixing \(\eps>0\), it follows from
Theorem \ref{thmo13} and Proposition \ref{o28.1}, as before, that there exists $t_1 < \infty$, 
not depending on $X(0),Y(0)$, $\omega$ or the particular limiting simple pursuit in $\Lambda_s$,
such that, for all $t\geq t_1$ and \(k=1,\ldots,m-1\),
\begin{align*}
\idist(\wt Z^{k,\infty}(t),\wt Z^{k+1,\infty}(t)) \leq \eps/(2(m-1)). 
\end{align*}
It therefore follows that, for some \(t_1>0\), \(n_0\), and all \(n>n_0\),
\begin{align*}
& \Prob{\idist( Z^{1,n}(t_1/n), Z^{m,n}(t_1/n)) \leq \eps} \quad = \quad 
 \Prob{\idist(\wt Z^{1,n}(t_1), \wt Z^{m,n}(t_1)) \leq \eps} 
\quad\geq\quad \\
& \Prob{ \idist(\wt  Z^{k,n}(t_1), \wt Z^{k+1,n}(t_1)) \leq \eps/(m-1) 
\text{ for } k=1, \ldots, m-1 } 
\quad>\quad 0\,.
\end{align*}
Consequently, by the Cameron-Martin-Girsanov theorem,
\begin{equation}\label{o28.4}
\Prob{
\idist( X(t_1/n), Y(t_1/n)) \leq \eps} \quad>\quad 0\,.
\end{equation}
This implies (\ref{o26.4d}) for $X(0)=x_0$, $Y(0)=y_0$, and arbitrary $x_0,y_0 \in \ol D$. 
\end{proof}

Theorem \ref{thmo24} states that shyness fails for CL domains.  
The proof requires establishing a uniform lower bound on the probability that shyness fails, 
over the different possible starting positions of \(X\) and \(Y\).
For this, we employ \citet[Proposition 20]{BramsonBurdzyKendall-2011},
which states the following.  Consider a (bounded) ESIC domain.
Suppose that \(\Prob{\inf_{0\leq t\leq t_1}|X(t)-Y(t)|\leq \eps}>0\) for some \(\eps>0\) and \(t_1>0\), for any coupled pair of 
Brownian motions \(X\) and \(Y\) with arbitrary starting points \(X(0), Y(0) \in D\). Then 
\begin{align}\label{o28.5}
\Prob{
\idist( X(t), Y(t)) \le \eps \text{ for some \(t\) with }{0 \leq t \leq t_1} } \quad\geq\quad p_1\,,
\end{align}
for some $t_1$ and $p_1>0$ not depending on $X(0)$ and $Y(0)$. The proof is 
based on an Arzela-Ascoli argument exploiting tightness of the processes and compactness 
of \(\overline{D}\).
 
\begin{proof}[Proof of Theorem \ref{thmo24}]
By  Proposition \ref{o26.4} and 
\citet[Proposition 20]{BramsonBurdzyKendall-2011}, \eqref{o28.5} holds true.
The remainder of the argument consists of an elementary iteration argument.
Consider processes $X$ and $Y$ starting from any pair of points in $\ol D$ and corresponding to any 
choice of $\JJ$ and $\KK$.
Because of the uniform bound in \eqref{o28.5}, the probability of $X$ and $Y$ not coming within
distance $\eps$ of each other on the interval $[kt_1, (k+1)t_1]$, conditional on not coming within this distance
before $kt_1$, is bounded above by $1-p_1$ for any $k$, by the Markov property. Hence, the probability
of $X$ and $Y$ not coming within distance $\eps$ of each other on the interval $[0, kt_1]$ is bounded above
by $(1-p_1)^k$. Letting $k\to \infty$, it follows that $X$ and $Y$ are not $\eps$-shy. Since $\eps$ can be
taken arbitrarily small, the proof is complete.
\end{proof}

\section{Domains with a stable rubber band, but no shy coupling}\label{sec:anexample} 
In this section, we exhibit a family of domains possessing stable rubber bands, but nevertheless
supporting no shy couplings.  Since these domains are not CL domains, these
examples complement Theorem \ref{thmo24}.
The family of domains is constructed by appending
to a domain possessing a stable rubber band another
(typically much larger) domain, so that the combined domain has the same 
stable rubber band but supports no shy coupling. The precise result is stated
in Theorem \ref{thm:rubber-band-not-shy}, which is the main result of the section.

%

For each of our examples,
we consider a bounded ESIC domain $D_1 \subset \mathbb{R}^d$, $d\geq 2$, 
that possesses a stable rubber band.
The larger domain is produced by appending a long thin cuboid to 
$D_1$. 
Some care needs to be taken to ensure that
the resulting domain still satisfies the uniform exterior sphere 
and uniform interior cone conditions, which requires us to impose some conditions on the boundary
of $D_1$. 

Rather than attempting to provide a more general result, for the sake of simplicity,
we suppose that there is a point \(p\) on the boundary \(\partial D_1\) such that, for some $r>0$, $\prt D_1 \cap \ball(p,r)$
is the graph of a $C^1$-function (in an appropriate orthonormal coordinate system), 
and that $D_1$ lies totally on one side of the hyperplane that is tangent to $\prt D_1$ at $p$; we further suppose that the distance from \(p\) to the stable rubber band is at least $2r$.
Translating and rotating the domain as necessary, we may suppose that the point \(p\) on the boundary is given by \((0,\dots,0,-a)\) for some $a\in(0, r/\sqrt{d} )$, and that the 
supporting hyperplane is \(\{x:x_d=-a\}\),
with
the open set \(D_1\) lying below this hyperplane.
We assume that $a$ is small enough so that 
\begin{align*}
\prt D_1 \cap \big((-a,a) \times \ldots \times (-a,a) \times (-2a, 0)\big)
\subset 
(-a,a) \times \ldots \times (-a,a) \times (-3a/2, 0).
\end{align*}
It is elementary to see that, for arbitrarily large $L$, there exists a domain $D$ such that 
\begin{enumerate}
\item
$D \cap (\{x:x_d<-a\} \setminus \ball(p,r)) 
=  D_1 \cap (\{x:x_d<-a\} \setminus \ball(p,r)) $,
\item
$\{x\in D : x_d >0\}  = 
(-a,a) \times \ldots \times (-a,a) \times (0, L)$,
\item
$(-a,a) \times \ldots \times (-a,a) \times (-2a, L)
\subset D$,
\item
\(D\) satisfies both the uniform exterior sphere 
and uniform interior cone conditions.
\end{enumerate}
It follows from the uniform exterior sphere and uniform interior cone conditions that
reflected Brownian motion on
$D$ is strong Markov, with normalized Lebesgue measure as its equilibrium probability measure 
(see, e.g., \citealp{BurdzyChen-1998}).

Heuristically speaking, $D$ is created by attaching 
a long thin cuboid $(-a,a) \times \ldots \times (-a,a) \times (-2a, L)$ to $D_1$ and smoothing the boundary
so that the sharp edges are only pointing outside the domain.
Note that \(L\)
can be increased arbitrarily without altering the construction close to \(D_1\).
Re-scaling the domain if necessary,
we may suppose that $a = 1/2$, and therefore that the intersection of $D$ with $\{x : x_d >0\}$ is $D_2 = (-1/2,1/2) \times \ldots \times (-1/2,1/2) \times (0, L)$.
We will assume that
\begin{align}\label{kbj1.1}
L \quad>\quad 512 d^2
\end{align}
and that
\begin{align}\label{eq:cuboid-size}
|\{x\in D: x_d < L/16\}| \quad<\quad |D|/8\,.
\end{align}

We now state the main result of the section.
\begin{thm}  
\label{thm:rubber-band-not-shy}
Suppose the domain $D$ is defined as above, by enlarging a given ESIC domain $D_1$ by 
appending a long cuboid.  This new domain \(D\) supports no shy co-adapted coupling
for reflected Brownian motion.  
\end{thm}

Before going into details, we describe the general plan for the proof of Theorem \ref{thm:rubber-band-not-shy}. Consider a coupling of two reflecting Brownian motions $X$ and $Y$ in \(D\).
For sufficently large \(t\), the two processes
\(X(t)\) and \(Y(t)\), when viewed separately, will be approximately in 
statistical equilibrium, and hence their marginal distributions will each approximate the normalized volume measure.
As a consequence of inequality \eqref{eq:cuboid-size}, it will follow 
(see Lemma \ref{lem:equilibrium}) that there is a positive probability of both 
\(X(t)\) and \(Y(t)\) lying in the part of the cuboid
$(-1/2,1/2) \times \ldots \times (-1/2,1/2) \times (L/16, L)$. 

Next consider the Lion and Man pursuit problem in the long cuboid $D_2$.
For each coordinate \(i\leq d-1\), 
we will produce a pursuit strategy 
given by a continuous vector field under which the Lion tracks the
Man closely in the coordinates \(1\), \ldots, \(i-1\), while approaching the Man in coordinate \(i\). This can moreover be done without the Man being able to move very much in the $d$th coordinate and, in particular,
before either the Lion or the Man leaves $D_2$
(see Lemma \ref{lem:iteration}).

A similar strategy (see Lemma \ref{lem:iteration-d}), but with respect to the
coordinate \(i=d\), results in the Lion approaching the Man in the $d$th coordinate 
while tracking the Man closely in the other \(d-1\) coordinates, and before either the Lion or 
the Man leaves $D_2$.

Employing this pursuit by the Lion of the Man, we will then argue, as in Section \ref{sec:shy}, 
that shyness must fail for the Brownian problem.

We now state and prove the three lemmas, in preparation of the proof of
Theorem \ref{thm:rubber-band-not-shy}.
\begin{lem}
\label{lem:equilibrium}
For large enough $u_0$,
all $(x,y)\in \overline{D}$, and any 
reflected Brownian motions $X$ and $Y$ on $D$ defined on the same probability space, with $X(0)=x$ and $Y(0)=y$,
\begin{equation}
\label{eqnsc2.1}
\Prob{X_d(u_0) \ge L/16 \text{ and } Y_d(u_0) \ge L/16} \geq 1/2 \,.
\end{equation}
\end{lem}
\begin{proof}
 As $t\rightarrow\infty$, the distributions of $X(t)$ and $Y(t)$ separately converge weakly to the equilibrium measure
on $\overline{D}$ of reflecting Brownian motion, which is normalized volume measure.
In fact (see \citealp[(2.2)]{BanBur}), for given $\eps \in (0,1]$, 
there exists \(u_0\) such that, for
all $x\in \overline{D}$ and any reflected Brownian motion $X$ on $D$ with $X(0)=x$,
the density of the distribution of $X$ at time $u_0$ is at most 
$(1 + \eps )/|D| \le 2/|D|$. 
The same remark applies to $Y$ and so, in view of \eqref{eq:cuboid-size},
\[
 \Prob{X_d(u_0)<L/16 \text{ or } Y_d(u_0)<L/16}
\leq 
2 (|D|/8) (2/|D|) = 1/2.
\]
The result follows by taking complements. 
\end{proof}

We now describe the pursuit strategies corresponding to each choice of coordinate \(i\leq d\)
by specifying continuous vector fields \(\chi^{(i)}(x,y)\) for the velocity of the Lion,
where \(x\) and \(y\) are the locations of the  Lion and of the Man. 
We allow the Man to choose any evasion strategy as long as his speed satisfies \(|y'(t)|\leq1\) for all \(t\).

We fix \(\delta\in(0,1)\), on which \(\chi^{(i)}(x,y)\) will depend implicitly;
in the proof of Theorem \ref{thm:rubber-band-not-shy}, we will let
$\delta \searrow 0$.
For \(i=1, 2,\ldots, d\), let \(\Pi_i\) be the orthogonal projection of \(\mathbb{R}^d\) onto 
the hyperplane defined by \(x_{i+1}=x_{i+2}=\ldots=x_d=0\). (\(\Pi_0\) is the trivial projection onto \(\{0\}\) and \(\Pi_d\) is the identity map.) 

We will define \(\chi^{(i)}(x,y)\) in three steps: first we will specify 
\(\Pi_{i-1}\chi^{(i)}(x,y)\) (equation \eqref{7.2'}),
then \((1-\Pi_i)\chi^{(i)}(x,y)\) (equation \eqref{eq:coord-i-plus-flat})
and finally 
\((\Pi_i-\Pi_{i-1})\chi^{(i)}(x,y)\) (equation \eqref{eq:coord-i-exact}).

Under the strategy given by \(\chi^{(i)}(x,y)\), we wish \(\Pi_{i-1}x\) to pursue \(\Pi_{i-1}y\) based on simple pursuit, but  
requiring \(\Pi_{i-1} x\) to move at speed at most \(\sqrt{1-\delta^2}\), 
and at a slower speed if \(x\) is close to \(y\) under the projection \(\Pi_{i-1}\).  Specifically, we set
\begin{equation}\label{7.2'}
 \Pi_{i-1}\chi^{(i)}(x,y) \quad=\quad 
 \min\left\{1,\frac{|\Pi_{i-1}(y-x)|}{\delta}\right\}\times \sqrt{1-\delta^2} \times \frac{\Pi_{i-1}(y-x)}{|\Pi_{i-1}(y-x)|}\,.
\end{equation}
Note that, as $\Pi_{i-1}(x-y) \rightarrow 0$, then $\Pi_{i-1}\chi^{(i)}(x,y) \rightarrow 0$.
Differentiating 
\(|\Pi_{i-1}(y-x)|\) with respect to $t$, it follows from (\ref{7.2'}) and the constraint \(|y'(t)|\leq1\) that,
when \(|\Pi_{i-1}(x-y)|\geq\delta\),
\begin{align}\label{jan10.3}
\left\langle
\Pi_{i-1}(y'-\chi^{(i)}(x,y)),
\frac{\Pi_{i-1}(y-x)}{|\Pi_{i-1}(y-x)|}
\right\rangle
\quad\leq\quad
1-\sqrt{1-\delta^2} 
\quad\leq\quad
\delta^2\,,
\end{align}
and therefore
the distance between \(\Pi_{i-1} x\) and \(\Pi_{i-1} y\) is either smaller than \(\delta\) or
increases only at rate at most \(\delta^2\).
(For  \(i=1\), we set \(\Pi_0=0\), in which case (\ref{jan10.3}) is vacuous.
Note that the bounds in (\ref{jan10.3}) do not depend on \((1-\Pi_i)\chi^{(i)}(x,y)\)
and 
\((\Pi_i-\Pi_{i-1})\chi^{(i)}(x,y)\),
which have not been defined yet.)

We set 
\begin{equation}\label{eq:coord-i-plus-flat}
 (1-\Pi_i)\chi^{(i)}(x,y) \quad=\quad 0\,,
\end{equation}
that is, the only nonzero components of $\chi^{(i)}$ are among its first \(i\) coordinates.


We still need to specify the $i$-th coordinate of $\chi^{(i)}$,  i.e., \(\chi^{(i)}_i(x,y)=(\Pi_i-\Pi_{i-1})\chi^{(i)}(x,y)\).  We define it so that it has the same 
sign as \(y_i-x_i=(\Pi_i-\Pi_{i-1})(y-x)\) and so that \(\chi^{(i)}(x,y)\) is a unit 
vector except when \(|y_i-x_i|\) is small.  Specifically,
\begin{equation}\label{eq:coord-i-exact}
 (\Pi_i-\Pi_{i-1})\chi^{(i)}(x,y) \;=\;
\min\left\{1,\frac{|y_i-x_i|}{\delta}\right\}\times\left(\sqrt{1-|\Pi_{i-1}\chi^{(i)}(x,y)|^2}\right) \times \operatorname{sgn}(y_i-x_i) \,.
\end{equation}
Because of (\ref{7.2'}), this implies that
\begin{align}
\label{7.3'}
|(\Pi_i-\Pi_{i-1})\chi^{(i)}(x,y)| \ge \delta \qquad \text{if } |y_i - x_i|\ge \delta.
\end{align}
Note that, as \(|x_i -y_i| \rightarrow 0\), then 
 $(\Pi_i-\Pi_{i-1})\chi^{(i)}(x,y)\rightarrow 0$.
On account of this and the observation after (\ref{7.2'}), it is not difficult to check 
that \(\chi^{(i)}(x,y)\) is continuous in \(x\) and \(y\). 

A crucial point in the strategy associated with
\(\chi^{(i)}(x,y)\), for given \(i<d\), is that it will force \(|x_i-y_i|\) 
to become small before \(y_d\) has the chance to decrease by more than a fixed amount
$4\gamma_i$ that is independent of the Man's strategy, where
\begin{equation}\label{eq:gamma-constant}
 \gamma_i \quad=\quad 1 \vee \frac{1}{\delta}|\Pi_{i-1}(x(0) - y(0))|. 
\end{equation} 

We will apply Lemma \ref{lem:iteration} in the probabilistic part of the argument, but the following explanation may help elucidate our inductive strategy. 
Heuristically speaking, at the $i$-th step, the lemma will be applied with the starting points $x(0)$ and $y(0)$ replaced by $x(u_{i-1})$ and $y(u_{i-1})$, 
and with the function $y(u_{i-1} + \,\cdot\,)$ in place of the function $y(\,\cdot\,)$.


\begin{lem}\label{lem:iteration}
 Choose $i\in \{2,\ldots,d-1\}$,  and assume $x(0)$ and $y(0)$  
lie in the long cuboid 
\(D_2 = (-1/2,1/2) \times \ldots \times (-1/2,1/2) \times (0, L)\), with $x(0)$ and $y(0)$
satisfying $x_d(0) > 0$, $y_d(0) > 4\gamma_i$.  
Assume that $x$ and $y$ move at unit speed or less, with the motion of $x$ being given by 
\(x'=\chi^{(i)}(x,y)\).  
There exists \(t<\infty\) at which \(|x_i(t)-y_i(t)| \le \delta\); denote by \(u_i\) the first time \(t\) at which this condition is satisfied.
Whatever the motion of \(y\), one has \(u_i\leq 1/\delta \).  Moreover,
%
\begin{align}
 |(\Pi_i-\Pi_{i-1})(x(u_i)-y(u_i))| \quad&=\quad |x_i(u_i)-y_i(u_i)| \quad\le\quad \delta\,, \label{eqn:iteration-1}\\
 (1-\Pi_i)x(t) \quad&=\quad (1-\Pi_i)x(0) \quad \text{for } t \leq u_i\,, \label{eqn:iteration-2}\\
 |\Pi_{i-1}(x(u_i)-y(u_i))| \quad&\leq\quad |\Pi_{i-1}(x(0)-y(0))| + 2\delta\,, \label{eqn:iteration-3}\\
 y_d(t) \quad&\geq\quad y_d(0) - 4 \gamma_i\quad \text{for } t \leq u_i \,. \label{eqn:iteration-4}
\end{align}
\end{lem}

Note that, in the case where $i=1$,  (\ref{eqn:iteration-3}) is vacuous and the other formulas
hold trivially with $u_1 \le 1$, since
the width of the first component of the cuboid is $1$
and $|\chi_1^{(1)}(x(t),y(t))| = 1$ for $t\le u_1$. 
%
 
\begin{proof}[Proof of Lemma \ref{lem:iteration}]
The formulas (\ref{eqn:iteration-1})--(\ref{eqn:iteration-4}) hold trivially, with $u_i=0$, when 
$|y_i(0) - x_i(0)| \le \delta$.  So, we will assume that $|y_i(0) - x_i(0)| > \delta$, with $u_i$ being the
time $t$ at which  $|y_i (t) - x_i (t)| = \delta$ first occurs.

Assume for the moment that (\ref{eqn:iteration-4}) holds, but with the weaker $t\le u_i^*$ in place of
$t\le u_i$, where $u_i^* = u_i \wedge (1/\delta)$.  Then, $x$ and $y$ both remain in the long cuboid
until time $u_i^*$.

 By (\ref{7.3'}), the speed of the component
$x_i$ is at least $\delta$, up until time $u_i^*$.
  Since the width of the $i$th component of the cuboid is $1$, it follows that
$u_i = u_i^* \le 1/\delta$.
Also, inequality \eqref{eqn:iteration-1} follows immediately from the definition of \(u_i\).

 Equation \eqref{eqn:iteration-2} follows from \((1-\Pi_i)\chi^{(i)}(x,y)= 0\).

Let $t_*$ be the supremum of $t\leq u_i$ such that
$|\Pi_{i-1}(x(t)-y(t))| \le \delta$; we let $t_*=0$ if there is no such $t$.
 Inequality \eqref{eqn:iteration-3} follows from the upper bound $\delta^2$ 
  in \eqref{jan10.3} on the directional
derivative of  \(|\Pi_{i-1}(x-y)|\) on the interval $[t_*, u_i]$,  
and from the bound \(u_i\leq 1/\delta\).

In order to complete the proof, it remains to demonstrate  (\ref{eqn:iteration-4}), 
 with $t\le u_i^*$ in place of $t\le u_i$.   The argument strongly uses
the definition of $\chi^{(i)}(x,y)$, which will ensure that the pursuit by the Lion of the Man is 
``efficient" with respect to the allowed change of the $d$th coordinate of the Man.   
The argument requires some  estimation since 
$\chi^{(i)}(x,y)$ is constructed in terms of the Euclidean metric, whereas we will need bounds with
respect to the L$^1$ metric in order to obtain  (\ref{eqn:iteration-4}).  
  

We 
choose $0=a_0<a_1<\ldots <a_J=1$ and let
$A_j = (a_{j-1}, a_j)$ such that, for any points
$x$, $\tilde{x}$, $y$, $\tilde{y}$ with 
$|\Pi_{i-1} (x- y)|, |\Pi_{i-1} (\tilde{x}- \tilde{y})|\in A_j$,
for given $j$,   
 \begin{equation}
 \label{eqnsc3.1'}
 |\chi_i^{(i)}(x, y)| - |\chi_i^{(i)}(\tilde{x}, \tilde{y})| \le \delta / 4.
 \end{equation}
Note that, on $t\le u_i^*$, $\chi_i^{(i)}(x(t),y(t))$ depends only on $|\Pi_{i-1}(x(t)-y(t))| \wedge \delta$.
 Let $B_j \subseteq [0,u_i^*]$ denote the time set on which $|\Pi_{i-1} (x(t)-y(t))| \in A_j$.  One can choose $a_j$ 
 so that the set where $|\Pi_{i-1}(x(t)-y(t))| = a_j$ has measure $0$ and so that 
$ a_{J-1} \in [\delta\gamma_i, 2\delta\gamma_i]$.
%
((\ref{eqnsc3.1'}) is satisfied on $[a_{J-1},a_J]$ since $a_{J-1} \ge \delta$, and so 
$\chi_i^{(i)}(x,y)$ is constant there.) 
We claim that
 \begin{equation}
 \label{eqnsc3.2}
 \begin{split}
 \int_{B_j} \left(\sqrt{\sum_{k=1}^{i-1}x_k^{\prime}(t)^2} - \sqrt{\sum_{k=1}^{i-1}y_k^{\prime}(t)^2}\right)\,dt
 &\le a_j - a_{j-1} \quad \text{for } j\le J-1, \\
 &\le 0 \quad \text{for } j=J,
 \end{split}
 \end{equation} 
 which we demonstrate at the end of the proof.

Employing the definition of $\chi_i^{(i)}$ and $|y^{\prime}(t)| \le 1$, we have
$$ x_i^{\prime}(t)^2 = 1 - \sum_{k=1}^{i-1} x_k^{\prime}(t)^2 \quad \text{and} \quad
y_d^{\prime}(t)^2 \le 1 - \sum_{k=1}^{i-1} y_k^{\prime}(t)^2 $$
for $t\le u_i^*$.
 On account of $1-v \le \sqrt{1-v} \le 1-v/2$ for $v\in [0,1]$, it follows from this and 
(\ref{eqnsc3.2})  that
 \begin{equation}
 \label{eqnsc3.3}
 \begin{split}
 \int_{B_j}\left(\frac{1}{2}y_d^{\prime}(t)^2 - x_i^{\prime}(t)^2\right)\,dt &\le a_j - a_{j-1} \quad \text{for } j\le J-1, \\
 &\le 0 \quad \text{for } j=J.
 \end{split}
 \end{equation}
Because of  (\ref{7.3'}) and (\ref{eqnsc3.1'}), for $t_1,t_2 \in B_j$, 
 \begin{equation}
\label{eqnsc3.3'}
\frac{1}{2} \quad\leq\quad
 \left(\frac{|x_i^{\prime}(t_1)|}{|x_i^{\prime}(t_2)|}\right)^2  
\quad\leq\quad 2\,. 
 \end{equation}
 Isolating the term $y_d^{\prime}(t)^2$ in (\ref{eqnsc3.3}), applying the Cauchy-Schwarz inequality to its integral, 
 applying (\ref{eqnsc3.3'}) and the inequality $\sqrt{v+w} \le \sqrt{v} + \sqrt{w}$ to the other side, and summing over
 $j=1,\ldots,J$ yields
 \begin{equation}
 \label{eqnsc3.4}
 \int_0^{u_i^*}|y_d^{\prime}(t)|\,dt \le 2\sum_{j=1}^J |B_j|\min_{t\in B_j}|x_i^{\prime}(t)| + 
 \sum_{j=1}^{J-1}\sqrt{2(a_j-a_{j-1})|B_j|}\,.  
 \end{equation}
Since $x_i^{\prime}(t)$ retains the same sign on $t\le u_i^*$,
 the first sum on the right side of (\ref{eqnsc3.4}) is at most $2\int_0^{u_i^*}|x_i^{\prime}(t)|\,dt \le 2$. 
Because $u_i^* \le 1/\delta$ and 
$a_{J-1} \le 2\delta\gamma_i$,
it follows from the Cauchy-Schwarz inequality that
the second term on the right is at most 
 \begin{equation*}
\sqrt{2u_i^*\sum_{j=1}^{J-1}(a_j-a_{j-1})} \le 2\sqrt{\gamma_i} \le 2\gamma_i.
 \end{equation*}
 Hence, $y_d(t) - y_d(0) \ge -2 - 2\gamma_i \ge -4\gamma_i$ for $t\le u_2^*$, as desired.

 We still need to demonstrate (\ref{eqnsc3.2}).  
First, note that since each $A_j$ is open, so is each $B_j$.
Let $B_j^{\eta}$ denote the subset of $B_j$ consisting of 
 the union of all open intervals in $B_j$ with length at least $\eta$, with $\eta \in (0, (a_j-a_{j-1})/2]$.    
 In order to show (\ref{eqnsc3.2}), it is sufficient to show its analog 
 \begin{equation}
 \label{eqnsc4.1}
 \begin{split}
 \int_{B_j^{\eta}} \left(\sqrt{\sum_{k=1}^{i-1}x_k^{\prime}(t)^2} - \sqrt{\sum_{k=1}^{i-1}y_k^{\prime}(t)^2}\right)\,dt
 &\le a_j - a_{j-1} \quad \text{for } j\le J-1, \\
 &\le 0 \quad \text{for } j=J,
 \end{split}
 \end{equation} 
 for each such $\eta$, because the integrands 
 are bounded. 

 We can assume that $B_j^{\eta} \neq \emptyset$ in (\ref{eqnsc4.1}).  We decompose $B_j^{\eta}$ into disjoint intervals
 $(b_{\ell},c_{\ell})$, $\ell = 1,\ldots,L$, with $b_{\ell}$ and $c_{\ell}$ increasing in $\ell$.  
It follows from the definition of $\Pi_{i-1}\chi^{(i)}(x,y)$ 
and differentiation of \(|\Pi_{i-1}(y(t)-x(t)|\)
that, for any $\ell\le L$, 
 \begin{equation}
 \label{eqnsc4.2}
 \int_{b_{\ell}}^{c_{\ell}} \left(\sqrt{\sum_{k=1}^{i-1}x_k^{\prime}(t)^2} - 
 \sqrt{\sum_{k=1}^{i-1}y_k^{\prime}(t)^2}\right)\,dt \le |\Pi_{i-1}(x(b_{\ell})-y(b_{\ell}))| - 
 |\Pi_{i-1}(x(c_{\ell})-y(c_{\ell}))|.
 \end{equation}
 We claim that 
 \begin{equation*}
 |\Pi_{i-1}(x(b_{\ell + 1})-y(b_{\ell + 1}))| = |\Pi_{i-1}(x(c_{\ell})-y(c_{\ell}))|,
 \end{equation*} 
 with  both equalling either $a_{j-1}$ or $a_j$: 
 these are endpoints of $A_j$, and the length of
any time interval during which
the distance between $\Pi_{i-1}x(t)$ and $\Pi_{i-1}y(t)$ crosses 
 $A_j$ must be at least $|A_j|/2 = (a_j-a_{j-1})/2$.  Hence, such an interval is included in
 $B_j^{\eta}$, because $\eta \le (a_j - a_{j-1})/2$.  This would contradict the definitions of
$b_{\ell +1}$ and $c_{\ell}$ if the projected distances between $x$ and $y$ were different at these two times.  
%
 
Summing $\ell$ over $1,\ldots,L$ in (\ref{eqnsc4.2}), the terms from the right side therefore telescope, and
 so the left side of (\ref{eqnsc4.1}) is at most
 \begin{equation}
\label{eqnsc4.3}
 |\Pi_{i-1}(x(b_{1})-y(b_{1}))| - |\Pi_{i-1}(x(c_L)-y(c_L))|.
 \end{equation}
Since the difference in (\ref{eqnsc4.3}) is dominated by $a_j-a_{j-1}$, 
the first line on the right side of (\ref{eqnsc4.1}) follows immediately. The second line 
of  (\ref{eqnsc4.1}) follows by noting that 
 $|\Pi_{i-1}(x(b_1)-y(b_1))| = a_{J-1}$, for $j=J$, 
since $|\Pi_{i-1}(x(0)-y(0))|\le a_{J-1}$ (by 
the definition of $a_{J-1}$),
and therefore the second term in (\ref{eqnsc4.3}) is at least as large as the first.  
 This completes the proof of the lemma.
\end{proof}

We note that the times \(u_i\), $i=1,\ldots, d-1$, in Lemma \ref{lem:iteration}, 
depend on the trajectory \(y\) taken by the Man.  
Since $u_i$ can be up to order $1/\delta$, $u_i$ might be larger than the length $L$ of 
the cuboid when $\delta$ is chosen close to $0$.   Although this could conceivably allow the
Man to escape from the cuboid before being approached by the Lion,
the bound on $y_d(t) - y_d(0)$ in  (\ref{eqn:iteration-4}) will allow us to show this will not occur.

We also obtain bounds for the case \(i=d\); these bounds are much easier to derive than the
corresponding bounds in Lemma \ref{lem:iteration}.
\begin{lem}\label{lem:iteration-d}
Assume
 \(x(0)\) and \(y(0)\) lie in the long cuboid 
\(D_2 = (-1/2,1/2) \times \ldots \times (-1/2,1/2) \times (0, L)\),
with $x(0)$ and $y(0)$ satisfying \(0<x_d(0)\leq y_d(0)\).
 Assume that \(x\) and \(y\) move at unit speed or less,
with the motion of \(x\) being 
given by \(x'=\chi^{(d)}(x,y)\).  
There exists \(t<\infty\) at which \(|x_d(t)-y_d(t)| \le \delta\); denote by \(u_d\) the first time \(t\) at which this condition is satisfied.
Then, whatever the motion of \(y\), one has
\(u_d\leq L/\delta\).  Moreover,
\begin{equation}
 |\Pi_{d-1}(x(u_d)-y(u_d))| \quad\leq\quad |\Pi_{d-1}(x(0)-y(0))| + (L+1) \delta\, \label{eqn:iteration-1-d}
\end{equation}
and 
\begin{equation}
\label{eqnsc5.4'}
x_d(0)\le x_d(t)\leq y_d(t) \quad  \text{for }   t\leq u_d.
\end{equation}
\end{lem}
\begin{proof}
The inequality (\ref{eqnsc5.4'}) follows immediately from $x_d(0) \le y_d(0)$ and the definition of $u_d$.
The inequality (\ref{eqn:iteration-1-d})  holds trivially when 
$x_d(0) \ge y_d(0) -\delta$, so we will assume that $x_d(0) < y_d(0) -\delta$, with
$u_d$ being the time  at which  $x_d(t) = y_d(t) -\delta$ first occurs.

Since \(\operatorname{sgn}(x_d')> 0\) over the time interval $[0,u_d)$, 
one has \(0< x_d<y_d-\delta\) there, and
%
\(x_d\) can travel no further than \(L-\delta  < L\) up until time \(u_d\). 
Also, by (\ref{7.3'}), the speed of the component $x_d$ is at least $\delta$ up until time $u_d$.   
%
Consequently, \(u_d\leq L/\delta\), as required.

Let $t_*$ be the supremum of $t\leq u_d$ such that
$|\Pi_{d-1}(x(t)-y(t))| \le \delta$; we let $t_*=0$ if there is no such $t$.
Inequality (\ref{eqn:iteration-1-d}) follows from the upper bound $\delta^2$ in (\ref{jan10.3})
on the directional derivative of $|\Pi_{d-1}(x-y)|$
on the interval $[t_*, u_d]$ , and on the bound $u_i \le L/\delta$.

\end{proof}
As before, \(u_d\) depends on the trajectory \(y\) taken by the Man.

We now outline the proof of Theorem \ref{thm:rubber-band-not-shy}.  The reasoning is similar
to that employed in the proofs of Proposition \ref{o26.4} and Theorem \ref{thmo24} 
in the previous section, 
where we employed the Lion and the Man problem to
demonstrate the absence of shy couplings for Brownian motion;
here, we will employ Lemmas \ref{lem:equilibrium}, 
\ref{lem:iteration} and \ref{lem:iteration-d} instead of Theorem \ref{thmo13}.
In the present setting, after employing Lemma \ref{lem:equilibrium},
we must piece 
together analogous results over $d$ time intervals,
and the roles of the Lion and the Man for the
two Brownian motions may need to be interchanged at the beginning of the last interval. 

\begin{proof}[Proof of Theorem \ref{thm:rubber-band-not-shy}]
Consider a pair of co-adapted reflecting Brownian motions on $\overline{D}$.  
By Lemma \ref{lem:equilibrium}, there is a nonrandom time $u_0$ such that, for any pair of
initial states $X(0)$ and $Y(0)$, 
\begin{equation}
\label{eqnsc6.1}
\Prob{X_d(u_0) \ge L/16 \text{ and } Y_d(u_0) \ge L/16}
\geq 1/2.
\end{equation}
Restarting the process at time $u_0$, we will apply (\ref{eqnsc6.1}), and Lemmas \ref{lem:iteration} and \ref{lem:iteration-d}
to deduce that, for any given
$\varepsilon \in (0,1)$,
\begin{equation}
\label{eqnsc6.2}
\Prob{\inf_{0\le s\le t}\text{dist}_{\bf I}(X(s),Y(s))\le \varepsilon} \quad>\quad 0
\end{equation}
for some $t$ not depending on $X(0)$ and $Y(0)$, where $\text{dist}_{\bf I}$ is the intrinsic distance metric on $\overline{D}$.
This is the analog of (\ref{o26.4d}).  
It is not hard to modify the argument in the proof of
\citet[Proposition 20]{BramsonBurdzyKendall-2011}
to show that \eqref{eqnsc6.2} implies
the uniform bound 
\begin{equation}
\label{eqnsc6.3}
\Prob{\inf_{0\le s\le t_{1}}\text{dist}_{\bf I} (X(s),Y(s))\le \varepsilon} \quad\ge\quad p_{1}\,,
\end{equation}
for some $t_{1}$ and $p_{1} > 0$ not depending on $X(0)$ and $Y(0)$. 
The uniform bound in (\ref{eqnsc6.3}) 
permits us to iterate the inequality \eqref{eqnsc6.3}  repeatedly, 
from which it follows that the coupling cannot be shy.

We now provide details for the derivation of (\ref{eqnsc6.2}).  
Consider an arbitrarily small $\delta\in(0,1)$. Assume that 
$X^n(0) = x(0)$, $Y^n(0)  = y(0)$ and
 $x_d(0),y_d(0) \ge L/16 $.
For specific stopping times \(U^i\), $i=0,\ldots,d-1$, to be defined below, with $U^0=0$ and 
\(U^i-U^{i-1}\in [0, 1/\delta]\), 
we let
\begin{equation*}
A_i = \Bigg\{|\Pi_i (X(U^{i}) - Y(U^{i}))|\le 4\delta i , \quad 
\inf_{U^{i-1} \leq t \leq U^i} X_d(t) \ge L/16 - i, \quad
\inf_{U^{i-1} \leq t \leq U^i} Y_d(t)\ge L/16 - 16i^2 \Bigg\}\,.
\end{equation*}
Note that it immediately follows from the first and third inequalities, and (\ref{kbj1.1}) that
\begin{equation}
\label{eqnsc6.3new}
Y_d(U^{i}) > 4\gamma_{i+1}^{\prime},
\end{equation}
where
$ \gamma_{i}^{\prime}= 1 \vee \frac{1}{\delta}|\Pi_{i-1}(X(U^{i-1}) - Y(U^{i-1}))|$.
We will show by induction that
\begin{align}
\label{eqnsc7.1n}
&\Prob{A_1} \quad>\quad0\,, \\
\label{eqnsc7.1}
&\Prob{A_i \;\Big|\; \bigcap_{k=1}^{i-1} A_{k}} \quad>\quad0\,, \qquad i=2, \dots, d-1\,.
\end{align}

We start with the case $i=1$, and 
define $(X^n(t),Y^n(t))$ and 
$(\tilde{X}^n(t),\tilde{Y}^n(t))$ as in (\ref{o26.20})-(\ref{o26.21}) and (\ref{o26.22})-(\ref{o26.23}), 
with 
$X^n(0) = x(0)$ and $Y^n(0)  = y(0)$,
and with 
$\chi$ replaced by \(\chi^{(1)}\) as defined before  Lemma \ref{lem:iteration}.  
The same reasoning as in the proof of Proposition \ref{o26.4}, but using Lemma
\ref{lem:iteration} instead of Theorem \ref{thmo13}, can be applied to analyze the limiting behavior of
$(\tilde{X}^n(t),\tilde{Y}^n(t))$ as $n\rightarrow\infty$.  The stopping time $T^n$ defined
below (\ref{o26.21}) is replaced by
the time at which either $X^n$ or $Y^n$ leaves $\overline{D}_2$. 
(We note that this means we can work throughout this proof with Euclidean
distance 
rather than intrinsic distance \(\text{dist}_{\bf I}\), 
since the two agree for pairs of points chosen within
the convex set $\overline{D}_2$.)
As in the proof of Proposition \ref{o26.4}, there exists 
a stopping time $\wt T^* \leq 1/\delta$ and
processes 
$\{\wt X^\infty(t), t\in [0,  1/\delta]\}$
and 
$\{\wt Y^\infty(t), t\in [0, 1/\delta]\}$, with
$\wt X^\infty(0)  = x(0)$, $\wt Y^\infty(0)=y(0) $, 
$\wt X^\infty(t)  = \wt Y^\infty(t) $  
for $ t \in [ \wt T^*, 1/\delta]$, $|\frac\prt {\prt t}
\wt Y^\infty(t) | \leq 1$ for $0 \leq t \leq 1/\delta$,
and
\begin{align*}
 \wt X^\infty(t) &= \int_0^t \chi^{(1)}(\wt X^\infty(s),\wt Y^\infty(s))\d s \qquad \text{  for  } t < 
\wt T^*\,,
\end{align*}
such that $(\wt X^n,\wt Y^n)$ converges a.s. to $(\wt X^\infty,\wt Y^\infty)$ 
uniformly on $[0,1/\delta]$.

Note that $\gamma_1^{\prime} = 1$; together with
$y_d(0) \geq L/16$ and \eqref{kbj1.1}, this implies $y_d(0) > 4 \gamma_1^{\prime}$, and so all of the 
conditions of Lemma \ref{lem:iteration} are satisfied. Applying the lemma, 
with $\wt X^\infty$ and $\wt Y^\infty$ in place of $x$ and $y$,
and denoting by $\wt U^1$ the first time $s$ at which
$|\wt X_1^\infty(s)-\wt Y_1^\infty(s)|\leq \delta$, it follows that $\wt U^1 \leq 1/\delta$. 
Moreover,
$|\Pi_{1}(\wt X^\infty(\wt U^1)-\wt Y^\infty(\wt U^1))|\leq \delta$ by  \eqref{eqn:iteration-1}. Since 
$\wt X^\infty(0)  = x(0)$ and  $x_d(0) \ge L/16 $, it follows from \eqref{eqn:iteration-2} that 
$\inf_{0 \leq t \leq \wt U^1}\wt X_d^\infty(t)  \ge L/16 $; it also follows from
\eqref{eqn:iteration-4} that $\inf_{0 \leq t \leq \wt U^1}\wt Y_d^\infty(t)  \ge L/16 -4 $.
These observations and the fact that 
$(\wt X^n,\wt Y^n)$ converges a.s. to $(\wt X^\infty,\wt Y^\infty)$ 
uniformly on $[0,1/\delta]$ imply that, for large enough \(n\),
\begin{equation*}
\P\Bigg[|\Pi_1 (\wt X^n(\wt U^{1}) - \wt Y^n(\wt U^{1}))|\le 4\delta  , \quad 
\inf_{0 \leq t \leq \wt U^1} \wt X^n_d(t) \ge L/16 -1,
\quad \inf_{0 \leq t \leq \wt U^1}\wt Y^n_d(t)\ge L/16 - 16 
\Bigg] > 0.
\nonumber
\end{equation*}
By the same argument as in \eqref{o28.2}, it follows that,
for some stopping time $U^1 \leq \wt U^1 \leq 1/\delta$,
\begin{equation*}
\P\Bigg[|\Pi_1 ( X( U^{1}) -  Y( U^{1}))|\le 4\delta  , \quad 
\inf_{0 \leq t \leq  U^1}X_d( t) \ge L/16 - 1, 
\quad \inf_{0 \leq t \leq  U^1} Y_d( t)\ge L/16 - 16 
\Bigg] > 0.
\nonumber
\end{equation*}
This completes the proof of \eqref{eqnsc7.1n}.

We will next present the induction step. Suppose that 
\eqref{eqnsc7.1n} and \eqref{eqnsc7.1} hold for $1$,  $2$, \dots, $i-1$.
We define $(X^n(t),Y^n(t))$ and 
$(\tilde{X}^n(t),\tilde{Y}^n(t))$ as in (\ref{o26.20})-(\ref{o26.21}) and (\ref{o26.22})-(\ref{o26.23}),
relative to the processes $\{X(U^{i-1} + \, \cdot\,)\}$ and
 $\{Y(U^{i-1} + \, \cdot\,)\}$ in place of $\{X(\, \cdot\,)\}$ and $\{Y(\, \cdot\,)\}$ (using \(\chi^{(i)}\) instead of $\chi$).
To simplify our presentation, we do not indicate in our notation that $X^n(t)$ and $Y^n(t)$ depend on $i$; the same remark applies to other processes and random variables used in the induction step.
Note that
$X^n(0) = X(U^{i-1})$ and $Y^n(0)  = Y(U^{i-1})$. 
Just as in the first step, we can find 
a stopping time $\wt T^* \leq 1/\delta$ and
processes 
$\{\wt X^\infty(t), t\in [0,  1/\delta]\}$
and 
$\{\wt Y^\infty(t), t\in [0, 1/\delta]\}$, with
$\wt X^\infty(0)  = X(U^{i-1})$, $\wt Y^\infty(0)=Y(U^{i-1}) $, 
$\wt X^\infty(t)  = \wt Y^\infty(t) $  
for $ t \in [ \wt T^*, 1/\delta]$, $|\frac\prt {\prt t}
\wt Y^\infty(t) | \leq 1$ for $0 \leq t \leq 1/\delta$,
and
\begin{align*}
 \wt X^\infty(t) &= \int_0^t \chi^{(1)}(\wt X^\infty(s),\wt Y^\infty(s))\d s \qquad \text{  for  } t < 
\wt T^*\,,
\end{align*}
such that $(\wt X^n,\wt Y^n)$ converges a.s. to $(\wt X^\infty,\wt Y^\infty)$ 
uniformly on $[0,1/\delta]$.

Assume that $\bigcap_{k=1}^{i-1} A_{k}$ holds.  Then, by (\ref{eqnsc6.3new}), 
$Y_d(U^{i-1}) > 4\gamma_i^{\prime}$.
We can therefore
apply Lemma \ref{lem:iteration}
to $\wt X^\infty$ and $\wt Y^\infty$ in place of $x$ and $y$.
Let $\wt U^i$ be the first time $s$ such that 
$|\wt X_i^\infty(s)-\wt Y_i^\infty(s)|\leq \delta$ and note that $\wt U^i \leq 1/\delta$ by Lemma \ref{lem:iteration}.

We are assuming that the conditioning event 
\(\bigcap_{k=1}^{i-1}A_k\)
in \eqref{eqnsc7.1} holds, so 
$|\Pi_{i-1}(X(U^{i-1}) - Y(U^{i-1}))|   \leq 4\delta (i-1)$. This and 
\eqref{eqn:iteration-3} imply that 
$|\Pi_{i-1}(\wt X^\infty(\wt U^{i}) - \wt Y^\infty(\wt U^{i}))|   \leq 4\delta (i-1) + 2 \delta$.
It follows from
\eqref{eqn:iteration-1} that 
$|(\Pi_i-\Pi_{i-1})(\wt X^\infty(\wt U^{i}) - \wt Y^\infty(\wt U^{i}))|   \leq \delta$
so, combining this with the previous estimate, we obtain
\begin{align}\label{kbj1.2}
|\Pi_i(\wt X^\infty(\wt U^{i}) - \wt Y^\infty(\wt U^{i}))|   \leq 4\delta (i-1) + 3 \delta = 4 \delta i -\delta.
\end{align}
From (\ref{eqn:iteration-2}) and the induction hypothesis,
we obtain 
$\inf_{0 \leq t \leq \wt U^i}\wt X_d^\infty(t) = X_d(U^{i-1})\ge L/16 -i +1$; also, 
in view of (\ref{eqn:iteration-4}), 
\begin{align*}
\inf_{0 \leq t \leq \wt U^i}
\wt Y_d^\infty(t) &\geq Y_d(U^{i-1}) - 4 \gamma_{i}^{\prime}
\geq L/16 - 16(i-1)^2  - 4 \gamma_{i}^{\prime}\\
& \geq L/16 - 16(i-1)^2  - 4\left(1 \vee 4(i-1) \right)
\geq L/16 - 16 i^2 + 1\,.
\end{align*}
These observations and the fact that 
$(\wt X^n,\wt Y^n)$ converges a.s. to $(\wt X^\infty,\wt Y^\infty)$ 
uniformly on $[0,1/\delta]$ imply that, for large enough $n$, 
\begin{equation*}
\P\Bigg[|\Pi_i (\wt X^n(\wt U^{i}) - \wt Y^n(\wt U^{i}))|\le 4\delta i , \quad 
\inf_{0 \leq t \leq \wt U^i}
\wt X^n_d(t)\ge L/16 - i, \quad \inf_{0 \leq t \leq \wt U^i}
 \wt Y^n_d(t)\ge L/16 - 16 i^2 \Bigg] > 0\,.
\nonumber
\end{equation*}
By the same argument as in \eqref{o28.2}, it follows that,
for some $U^i \in [U^{i-1}, U^{i-1} + \wt U^i]$,
\begin{equation*}
\P\Bigg[|\Pi_i (X(U^{i}) - Y(U^{i}))|\le 4\delta i , \quad 
\inf_{U^{i-1} \leq t \leq U^i} X_d(t) \ge L/16 - i, \quad
\inf_{U^{i-1} \leq t \leq U^i} Y_d(t)\ge L/16 - 16i^2 \Bigg] > 0\,.
\nonumber
\end{equation*}
This completes the proof of \eqref{eqnsc7.1}.

Application of  \eqref{eqnsc7.1n} and \eqref{eqnsc7.1} with $i=2,\ldots,d-1$
yields, for $X(0) = x(0)$, $Y(0) = y(0)$ with $x_d(0),y_d(0) \ge L/16$, that
\begin{align}
\label{eqnsc7.2}
\P\Big(&|\Pi_{d-1} (X(U^{d-1}) - Y(U^{d-1}))|\le 4\delta (d-1) , \\
&\inf_{0 \leq t \leq U^{d-1}} X_d(t) \ge L/16 - d +1, \quad
\inf_{0 \leq t \leq U^{d-1}} Y_d(t)\ge L/16 - 16(d-1)^2 
\Big) >0. \nonumber
\end{align}

Our final step is very similar to the inductive step presented above but requires some minor modifications, 
where we apply Lemma \ref{lem:iteration-d},
in place of Lemma \ref{lem:iteration},
to processes $\wt X^\infty$ and $\wt Y^\infty$ constructed from the processes $\{X(U^{d-1} + \, \cdot\,)\}$ and
 $\{Y(U^{d-1} + \, \cdot\,)\}$. One of the assumptions of Lemma \ref{lem:iteration-d} is $0 < x_d(0) \leq y_d(0)$ whereas, in our setting,  
$X_d(U^{d-1}) \leq Y_d(U^{d-1})$ need not hold. 
To deal with the situation where $X_d(U^{d-1}) > Y_d(U^{d-1})$, 
we relabel the Lion and the Man in the Lion
and Man problem, exchanging the roles of $X$ and $Y$ in this step if necessary,
so that $L/16 - 16(d-1)^2 \le X_d(U^{d-1}) \leq Y_d(U^{d-1})$ holds.

Similar reasoning to the inductive step presented above,
together with  Lemma \ref{lem:iteration-d} in place of  
Lemma \ref{lem:iteration}, shows that there is a stopping
time $U^d$, with $U^d - U^{d-1} \leq L/\delta$,
such that
\begin{align*}
\mathbb{P}\Big[|X(U^{d}) - Y(U^{d})| &\leq 4\delta (d-1)  +\delta (L+3) \,, \\
L/16 - 16(d-1)^2 &-1 \leq  X_d(t)\leq Y_d(t) + 1, \text{  for } U^{d-1} \leq t \leq U^d \, \,|\, \bigcap_{k=1}^{d-1} A_{k}\Big] \quad > \quad 0.
\end{align*}
Combining this with 
\eqref{eqnsc7.2} implies that 
\begin{align}
\label{eqnscend}
\Prob{| X(U^{d}) - Y(U^{d})| < \delta (4d+L) , 
\quad\inf_{0 \leq t \leq U^{d}} X_d(t) >0, 
\quad\inf_{0 \leq t \leq U^{d}} Y_d(t)>0 
}\quad >\quad 0\,. 
\end{align}
Note that $U^d \leq (L+d)/\delta$.

We now complete the proof of the theorem.
Combining (\ref{eqnsc6.1}) with (\ref{eqnscend}), the lower bound in \eqref{eqnsc6.2} follows
upon setting $t = u_0 + (L+d)/\delta$ and $\delta = \varepsilon / (4d + L)$.
%
\end{proof}

\section{Various examples}\label{sec:variousexamples}
In this section, we present a number of examples involving CL domains and domains with
rubber bands.  Since we will be interested only in domains that satisfy the 
uniform exterior sphere and uniform interior cone conditions in this section,
we will implicitly assume that all domains discussed here satisfy these boundary regularity conditions.

\begin{example}\label{o24.1.i}
An example of a simply-connected domain that is not a CL domain
and yet for which 
all loops are contractible.
\rm
Let $D\subset \Reals^3$ be the interior of the intersection of the upper half-space $z\ge0$ with the spherical shell
$\ball(0,2) \setminus \ball(0,1)$. 
Loops in \(\ol D\) that do not lie wholly on $\prt \ball(0,1)$ can be contracted in \(\ol D\) along rays emanating from $(0,0,0)$ to smaller loops that lie wholly on $\prt \ball(0,1)$.
Loops in \(\ol D\) that lie on $\prt \ball(0,1)$ can be contracted in \(\ol D\) to the point $(0,0,1)$ along
great circles passing through $(0,0,1)$ and perpendicular to the boundary of the upper half-space. 
So, all loops in \(\ol D\) are contractible.
For an example of a rubber band in \(\ol D\) that is not well-contractible, 
consider the intersection $\lp$ of $\prt \ball(0,1)$ with the boundary of the upper half-space. 
Suppose that 
$\lp'\in \Lp$ and $d_H(\lp, \lp') \leq \eps$. Let $\lp''$ be the radial projection of $\lp'$
onto $\prt \ball(0,1)$. It is easy to check that $d_H(\lp, \lp'') \leq \eps$, and therefore
$\ell_{\lp'} \geq \ell_{\lp''} \geq \ell_\lp - c \eps^2$.
Hence, no contraction of $\lp$ satisfies Definition \ref{def:rubband} (b).
\end{example}

\begin{example}\label{o24.1.ii}
Star-shaped domains are CL domains.
\rm 
Suppose that $D$ is star-shaped, that is, for some $z_0\in D$ and all $z\in D$, the line segment between $z_0$ and $z$ is contained in $D$.
Consider any $\lp \in \Lp$ and let
$T_a(z) = z_0 + a(z-z_0)$.
Then, for $t,\gamma \in[0,1)$, 
$H(e^{2\pi i t},\gamma)= T_{1-\gamma} (\lp(t\ell_{\lp}))$ defines a contraction
$\{T_{1-\gamma}K\}_{\gamma \in [0,1)}$ of $\lp$.
Elementary calculations based on scaling show that this contraction satisfies
Definition \ref{def:rubband} (b).  So $D$ is a CL domain.
\end{example}

\begin{example}\label{jan9.2}
$\CAT0$ domains are CL domains.
\rm 
To see this, first note that a given loop can be approximated as closely as desired by a polygon. Choose any
fixed point $x_0\in D$, which will serve as our reference point.  Employing $x_0$ and the endpoints of any of the
line segments defining the polygon, since $D$ is assumed to be $\CAT0$, there is a unique pair of geodesics from  $x_0$ to these endpoints,
$\gamma_1 = \{\gamma_1(t), 0 \leq t \leq t_1\}$ and $\gamma_2 = \{\gamma_2(t), 0 \leq t \leq t_2\}$.
Moreover,
there exist geodesics in $\Reals^2$ (line segments),
$\wt \gamma_1 = \{\wt \gamma_1(t), 0 \leq t \leq t_1\}$ and $\wt\gamma_2 = \{\wt\gamma_2(t),  0 \leq t \leq t_2\}$, 
with
$\dist(\wt\gamma_1(t_1), \wt\gamma_2(t_2))=\intrinsicdist(\gamma_1(t_1), \gamma_2(t_2))$, and such that
$\intrinsicdist(\gamma_1(at_1), \gamma_2(at_2))\leq \dist(\wt\gamma_1(at_1), \wt\gamma_2(at_2))$
for $a\in [0,1]$.  These
geodesics induce a well-contractible homotopy to the point $x_0$, with a contractibility constant $c>0$ 
that is at least as large as that corresponding to a planar convex domain with the same 
diameter.
By
selecting a sequence of polygons that converges uniformly to the given loop and taking limits, one obtains
a well-contractible homotopy, with the same contractibility constant $c$, for the original loop.
This implies $D$ is a CL domain.
Note that star-shaped domains are not necessarily $\CAT0$ and $\CAT0$ domains are not necessarily star-shaped.
\end{example}

\begin{example}\label{jan9.3}
Construction of CL domains by modification of $\CAT0$ domains, and a further generalization.
\rm
For a given $\CAT0$ domain $D$, choose a point $x_0\in D$ as its reference point.  For each $x\in D$, there
exists a unique geodesic $\Gamma_x$ from $x_0$ to $x$; denote by $t_x$ the 
value of the parameter
at which $\Gamma_x(t_x)=x$. 
There are many ways in which one can truncate these geodesics at $s_x\le t_x$ so that the remaining region consisting
of the portions of the geodesics $\Gamma_x(t)$, with $t\in [0,s_x)$, is open and connected, and hence a domain.   The restricted domain $D_1$
thus defined will still be CL domain, but need not be $\CAT0$.  

An illustration is given in
Fig.~\ref{fig:twodumbbells}, where the non-$\CAT0$ domain on the left is obtained from the $\CAT0$ domain on the
right by making a shallow ``dent" on the end of one of its spheres. 
To see that the above recipe works in this case, choose the reference point $x_0$
to be the point inside the domain but on the boundary of the sphere,
at the spot antipodal to the center of the dent.
If the dent is shallow, then all geodesics from $x_0$ to all points
of the sphere, with the dent removed, are line segments.
In other words, the sphere with the dent removed is star-shaped
relative to $x_0$.
 Note that new domain is not
$\CAT0$, since it has a point on the surface where both principal curvatures are negative.

\begin{figure}[twodumbbels]
\centering
\includegraphicsKB[width=4in]{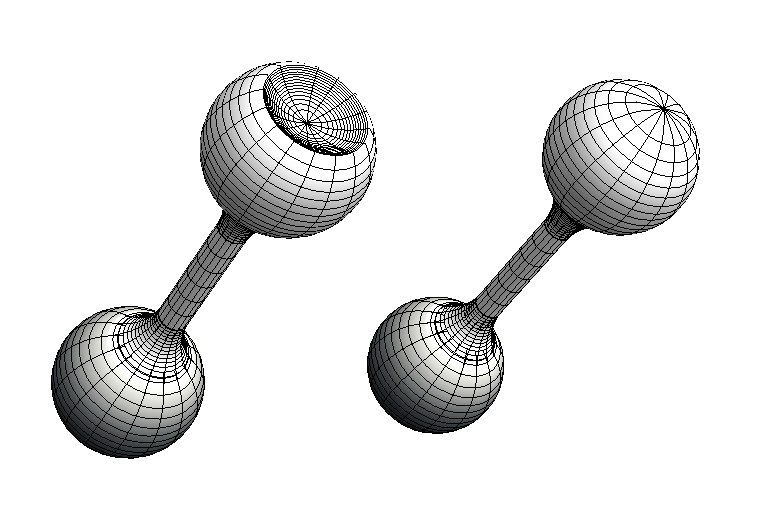}
\caption{The domain on the right that has dumbbells shape is $\CAT0$. The domain on the left
is obtained from the domain on the right by making a shallow spherical dent. 
The domain on the left is not $\CAT0$ but it is a CL domain.}
\label{fig:twodumbbells}
\end{figure}

\begin{figure}[thbp]
\centering
\includegraphicsKB[width=4in]{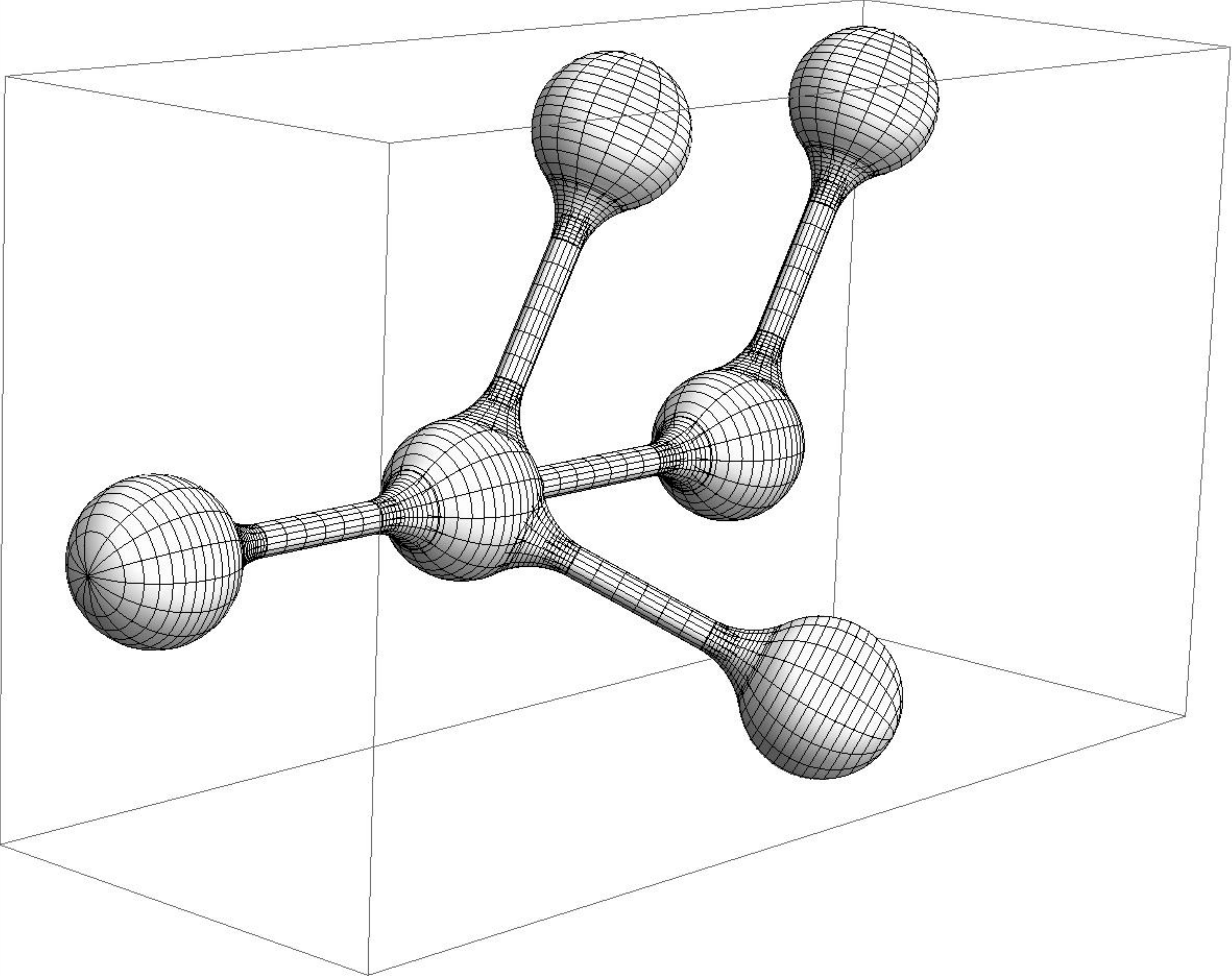}
\caption{This example of a $\CAT0$ domain appeared in \citet{BramsonBurdzyKendall-2011} where shy couplings and pursuit problems were analyzed in $\CAT0$ domains. }
\label{fig:rubdumb}
\end{figure}

A more general family of examples that is related to the previous one can be obtained by considering the $\CAT0$ domain given in \citet{BramsonBurdzyKendall-2011} and reproduced here in Fig.~\ref{fig:rubdumb}. 
Making one or more shallow dents 
in the spheres at the ends of the domain produces a new domain $D_1$ that is a CL domain, but not 
$\CAT0$, for the same reasons as before.   

Any star-shaped domain $D_1$ can be realized by applying the construction at the beginning of the example, and
choosing the domain $D$ to be any 
convex domain
containing $D_1$.  The domain $D$ is $\CAT0$ and its geodesics are the
line segments connecting pairs of points.  So, Example \ref{o24.1.ii} is included in Example \ref{jan9.3}.

We note that, for this construction, 
the $\CAT0$ property for the domain $D$ was only employed to ensure that 
pairs of geodesics emanating from the given reference point $x_0$ satisfy the $\CAT0$ property; the behavior of
other geodesics was not employed.   The reasoning in the first paragraph
thus extends to domains $D$ 
in which the reference point is connected by a single geodesic in $\ol D$ to any point in $\ol D$,
and for which pairs of geodesics
starting at the reference point satisfy the $\CAT0$ property.  The construction given in Example \ref{jan9.3} can therefore
also be viewed as a natural generalization of that given in Example \ref{o24.1.ii}.
\end{example}

\begin{example}\label{o24.1.vi}
There are domains with semi-stable rubber bands that are not stable.
\rm
\begin{figure}[htbp]
\centering
\includegraphicsKB[width=3in]{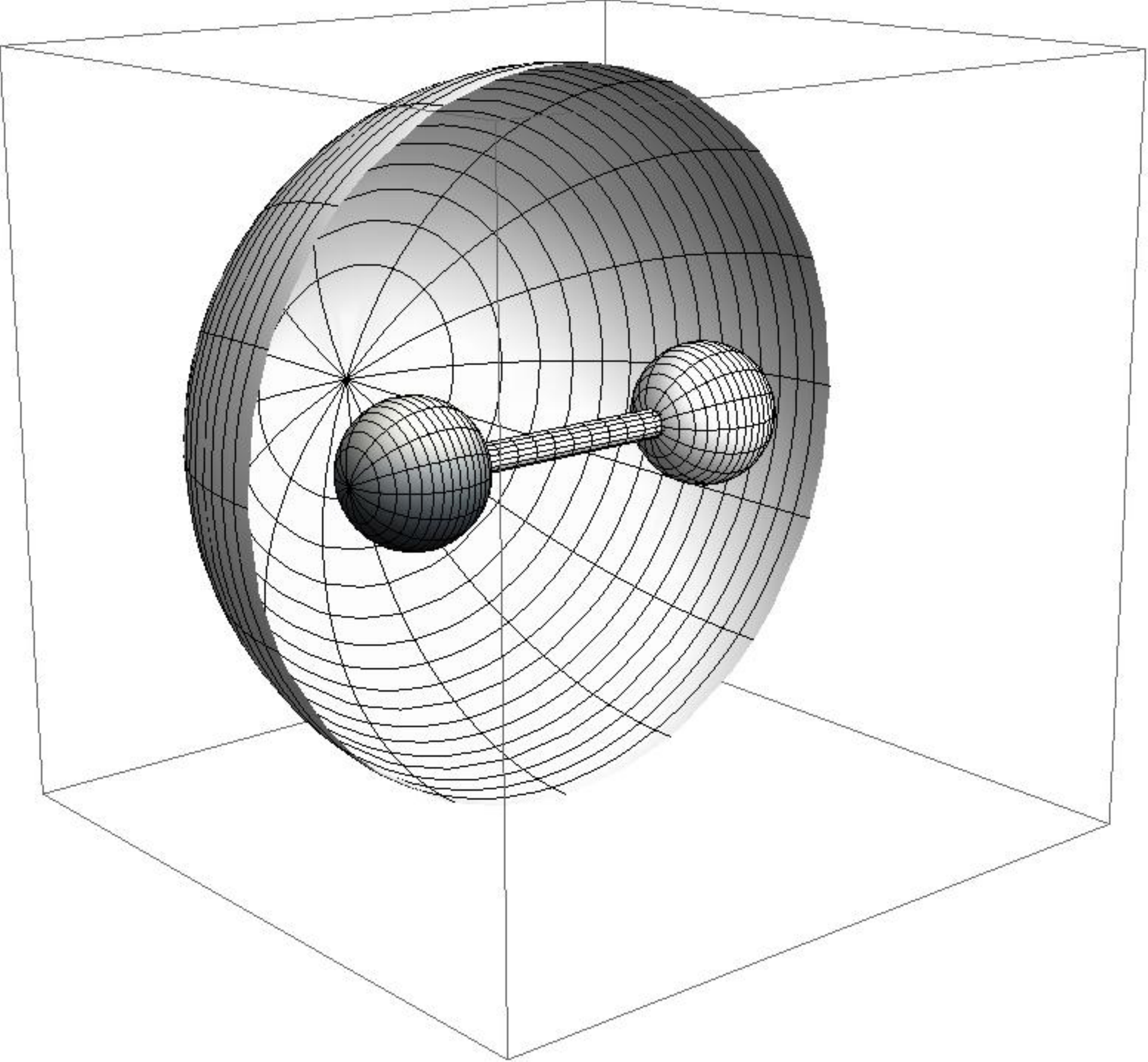}
\caption{``Dumbbells domain'' with a semi-stable rubber band and no stable rubber band.
One half of the outer boundary
is cut away to show the interior part of the boundary. Drawing is not to scale.}
\label{fig:rub1}
\end{figure}
Consider a domain $D$ that is a ball from which a dumbbell has been removed (see Figure \ref{fig:rub1}):
\begin{align*}
D = \ball(0,100) \setminus\left(
\ball((-10,0,0),2) \cup \ball((10,0,0),2)
\cup \bigcup_{-10\leq a \leq 10} \ball((a,0,0),1)
\right)\,.
\end{align*}
The loop $\{(0,\cos t, \sin t), 0 \leq t < 2\pi\}$ is semi-stable. 
It is not a stable rubber band.
\end{example}

\begin{example}\label{o24.1.vii}
If a bounded domain is not simply connected, then it must possess at least one semi-stable rubber band.
\rm
Homotopies preserve the homotopy class of a loop. If a domain fails to be simply connected, then there exists a non-trivial homotopy class of loops
and, within this class, there will be at least one loop minimizing the length function. This loop will be semi-stable, though not necessarily stable.
\end{example}

\begin{example}\label{o24.1.v}
An explicit construction of a domain with a stable rubber band.
\rm
\begin{figure}[bhtp]
\centering
\includegraphicsKB[width=2.5in]{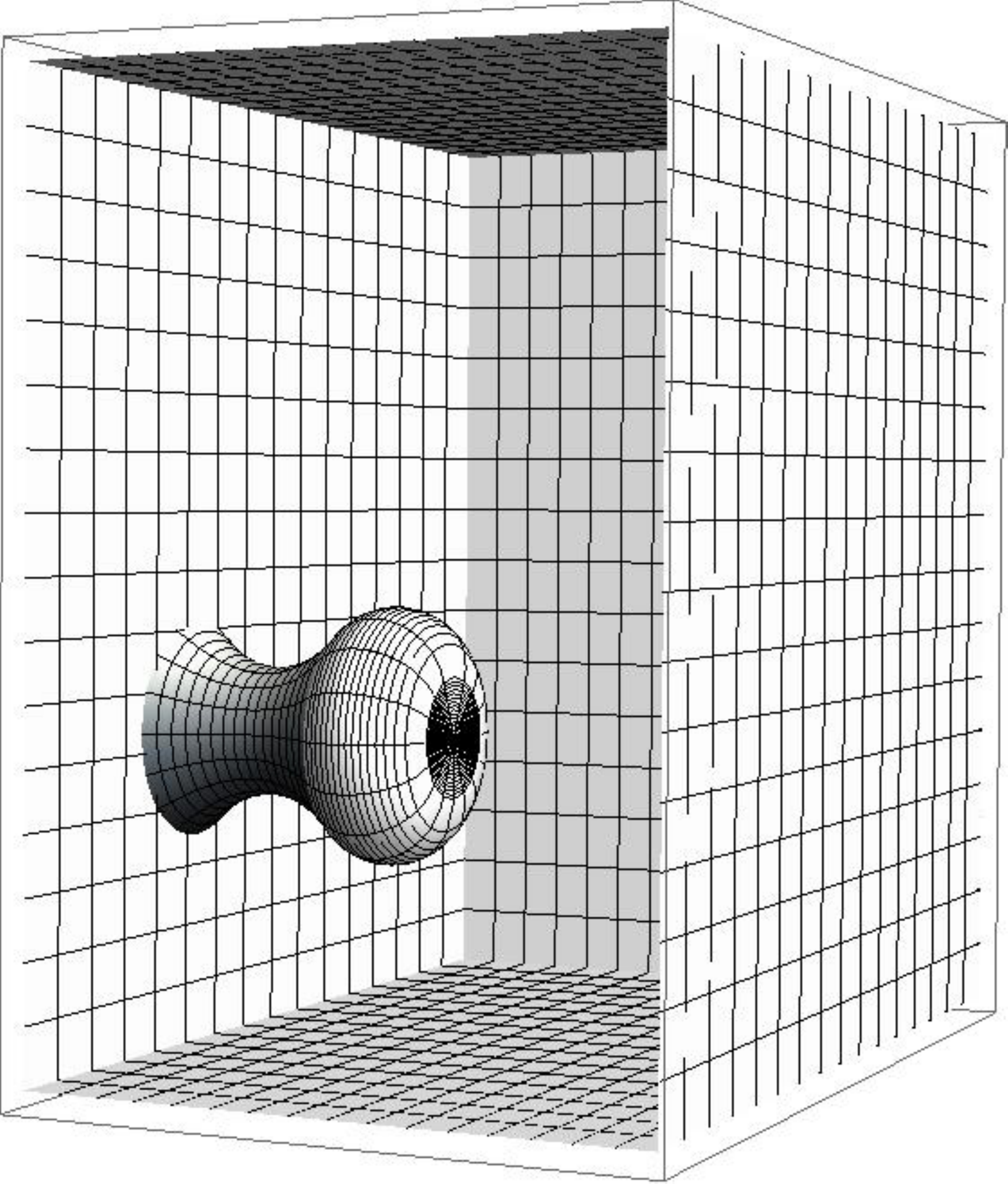}
\caption{``Doorknob domain'' with a stable rubber band.
Part of the outer boundary
is cut to show the interior part of the boundary. Drawing is not to scale.}
\label{fig:rub2}
\end{figure}
The essence of this construction is given in Figure \ref{fig:rub2}: a doorknob is attached to the interior of a box, and the stable rubber band fits around the neck of the doorknob.
Let
\begin{align*}
D &= ((0,10) \times (-10,20)^2)\\
&\setminus \Big(
\{(x,y,z) \in \Reals^3: y^2 + z^2 \leq 1 + (x-1)^2,\  0\leq x \leq 2\} \\
& \qquad \cup \{(x,y,z) \in \Reals^3: y^2 + z^2 \leq  - 4 (x-3) (x- 3/2),\  2 \leq x \leq 3\} \Big).
\end{align*}
The boundary of $D$ is not smooth, although edges can be smoothed without affecting the stable rubber band defined below.
We will argue that the following loop is a stable rubber band,
\begin{align*}
\lp(t) = (1, \cos(t), \sin(t)), \qquad 0 \leq t < 2\pi.
\end{align*}
Let $\eps=1/10$ in Definition \ref{o19.1} and fix some
$0<\eta \leq\eps$.
Suppose that
$\lp_1\in \Lp$ is such that $d_H(\lp, \lp_1)\geq \eta$ and, for some $n\geq 1$, there exists a continuous mapping $H: \SS \times [0,1] \to \ball(\lp,\eps)$
satisfying the conditions in Definition \ref{o19.1}. It 
is evident
that, for some $\delta>0$ depending only on $D$ and $\eta$, 
the length of the projection $\wh \lp_1$ of $\lp_1$ on the plane 
$A=\{(x,y,z): x=0\}$ 
 has length greater than $2\pi n +\delta$. 
Thus \(\lp\) is a stable rubber band in \(D\).
\end{example}

Physical intuition suggests that a ``typical domain" either contains a semi-stable rubber band or the domain is a CL domain, although it is easy
to construct domains that satisfy neither condition.  An intuitive justification is based on part (iv) of Definition \ref{def:rubband}, since
a ``typical function" on a compact interval is either non-monotone or its slope has one sign and is bounded away from $0$.   We close this
section with Conjecture \ref{conj:nowhere-dense}, which makes this claim precise.  The conjecture employs the following definition.


\begin{defn}\label{def:domain-space}
Consider the family $\mathbf{D}$ of all open bounded non-empty sets $D\subset \Reals^d$ with smooth boundary and $\kappa(D) <\infty$, 
where $\kappa(D)$ is the supremum over all $z\in \prt D$ of 
the absolute values of the principal curvatures at $z$.
The Gromov-Hausdorff distance induces a topology on this family.
Let
$\DD$ be the topological space of all pairs $(D, \kappa(D))$, 
$D\in \mathbf{D}, \kappa(D) >0$,
equipped with the product topology.
Let $\DD_s$ be the set of $(D, \kappa(D))$ such that $D$ contains a 
 semi-stable 
rubber band and
let $\DD_c$ be the set of $(D, \kappa(D))$ such that $D$ is a CL domain.
\end{defn}

Conjecture \ref{conj:nowhere-dense} states that the set of domains with semi-stable rubber bands and 
the set of CL domains are each open in the above topology, and that
the set of domains remaining, after removing these two sets, is nowhere dense.

\begin{conjecture}\label{conj:nowhere-dense}
The sets $\DD_s$ and $\DD_c$ are open in $\DD$. 
The set $\DD \setminus (\DD_s \cup \DD_c)$ is nowhere dense.
\end{conjecture}

\section{Acknowledgments}

We are grateful to Stephanie Alexander, Dick Bishop and Yu Yuan for their most useful advice.
We thank  Chanyoung Jun for sending us his interesting Ph.D. thesis \citep{JunPhD}.

\bibliographystyle{chicago}
\bibliography{rubberband}

\appendix   
\section{Uniform exterior sphere and uniform interior cone conditions imply $\CAT\kappa$}\label{app:reach}
In this appendix,  we will employ a theorem from \citet{Lytchak-2004}, to show that
the uniform exterior sphere and uniform interior cone conditions
imply the \(\CAT\kappa\) property, that is, ESIC domains are
\(\CAT\kappa\), for some \(\kappa>0\).
In order to establish this, it is useful to employ the following definition.
\begin{defn}[Lipschitz domain] \label{def:lipschitz-domain}
Recall
that a function $f: \Reals^{d-1} \to \Reals$ is
\emph{Lipschitz, with constant $\lambda < \infty$,} if $|f(v) - f(z)| \leq \lambda
|v-z|$ for all
$v,z \in \Reals^{d-1}$.
A domain \(D\in \mathbb{R}^d \) is \emph{Lipschitz, with
constant $\lambda$,} if there exists $\delta >0$ such that, for
every $v \in \prt D$, there exists an orthonormal
basis $e_1, e_2, \ldots, e_d$
and a Lipschitz function $f: \Reals^{d-1} \to
\Reals$, with constant $\lambda$, such that
\[
\{w \in \ball(v,\delta) \cap D\} \quad=\quad
\{w\in\ball(v,\delta): f(w_1, \dots, w_{d-1})<w_d\}\,,
\]
where we write $w_1=\langle w,e_1\rangle$, \ldots, $w_d=\langle w,e_d\rangle$.
\end{defn}
As noted in \citet[Section 2]{BramsonBurdzyKendall-2011},
Definition \ref{def:lipschitz-domain}  is equivalent to the  uniform interior cone
condition \ref{def:UICC} , with \(\lambda=\cot\alpha\).  Moreover,
if either holds for a given \(\delta>0\), then both hold for that \(\delta\) and all smaller \(\delta\).

Verification of the  \(\CAT\kappa\) property for ESIC domains is most easily done by first establishing
that the uniform exterior sphere condition implies a property known as \emph{positive reach}. 
We begin by stating its definition (adapted from
\citealp{Lytchak-2004}), and then sketch the \(\CAT\kappa\) implication in Lemma \ref{n7.1} and Corollary \ref{n7.2}.
\begin{defn}
A set $A\subset \Reals^d$ \emph{has positive reach greater than or equal to $r$} if, for
all $z \in \Reals^d$ with $\dist(z,A) \leq r$, there is a unique point $v\in A$ with $\dist(z,v) = \dist(z,A)$.
\end{defn}

\begin{lem}\label{n7.1}
Suppose that $D$ is a bounded domain that satisfies a uniform exterior sphere
condition and a uniform interior cone
condition. 
Then $\ol D$ has positive reach. 
\end{lem}

\begin{proof}
Suppose \(r\) is the uniform exterior ball radius and \(\alpha\in(0,\pi/2]\) is the uniform interior cone angle.
We shall show that the reach is at least as large as the minimum of \(r\sin\alpha\) and 
a positive constant \(\delta\) relating to the uniform interior cone condition. Here, \(\delta > 0\)
is chosen small enough so that, in any ball of radius \(\delta\), we may implement the uniform interior cone condition
with a fixed axis \(e_d\); moreover, we shall choose \(\delta\) small enough so that the cones extend sufficiently far 
so, within the ball, the domain \(D\) may be described as
the super-level set of a Lipschitz function, with Lipschitz constant \(\cot\alpha\), 
as given in Definition \ref{def:lipschitz-domain}. 

If the reach \(s'\) is smaller than \(\delta\), then there must exist \(s\geq s'\) and arbitrarily close to \(s'\)
such that there exists $z\in \Reals ^d$, with $\dist(z, \ol D)=s$, 
so that, for
distinct points $v_1, v_2 \in \prt D$, $\dist(z,v_1) = \dist(z,v_2) = \dist(z, \ol D) = s$. 
 We shall show that this will imply
\(s\geq r \sin\alpha\). From this, it will follow that \(D\) has positive reach at least as great 
as \(\min\{\delta, r\sin\alpha\}\).

Suppose that $e_d$ is the \(d^\text{th}\) vector in the orthonormal basis
corresponding to points in \(\ball(z,\delta)\) as in Definition \ref{def:lipschitz-domain}, noting that we have chosen \(s\) small enough so that there is a single Lipschitz function representation within \(\ball(z,\delta)\) based on \(e_d\).
Let \(M\) be a \(2\)-plane intersecting \(D\) and containing \(v_1\), \(v_2\), and \(v_1+e_d\). We note in passing that \(M\)
will also contain \(v_2+e_d\). As a consequence of \citet[Lemma 11]{BramsonBurdzyKendall-2011},
any point 
in \(M\cap \prt D\cap\ball(z,\delta)\)
 must be supported by an open disk in \(M\) of radius \(r\sin\alpha\), such that the disk and \(M\cap D\) are disjoint.

The ball \(\ball(z,s)\) intersects \(M\) in an open disk \(C_1\) of radius at most \(s\); moreover, 
it follows from their definition that \(v_1\) and \(v_2\) must lie on the boundary of \(C_1\). Furthermore, we may use the uniform interior cone condition (based locally on \(e_d\)) to argue that the line \(\ell\) through \(v_1\) and \(v_2\) must separate (in \(M\)) the center of \(C_1\) from the points \(v_1+e_d\), \(v_2+e_d\).

Let $v_3 \in M \cap \prt D \cap\ball(z,\delta)$ be the point with the same first $d-1$ coordinates as $(v_1+v_2)/2$.
(By the choice of \(\delta\), there will be exactly one such point.)
As we have noted above, $v_3$ lies on the boundary of an open disk $C_2$ with radius $r \sin \al$, 
which is disjoint from $\ol D$. 
Moreover, \(v_3\) must be separated in \(M\) from the center of \(C_1\) by \(\ell\), 
since otherwise \(C_1\) will intersect with \(D\).

We next argue using plane geometry as follows. Consider the case in which the disk \(C_2\) lies on one side of \(\ell\). Then the Lipschitz representation of \(D\cap \ball(z,\delta)\) implies that it is constrained to lie between the lines \(v_1+\mathbb{R}e_d\)
and \(v_2+\mathbb{R}e_d\), and hence has diameter no greater than \(\dist(v_1,v_2)\leq 2s\).
So, in this case, we can deduce that \(s\geq r\sin\alpha\).

\begin{figure}[thbp]
\centering
\includegraphicsKB[width=4in]{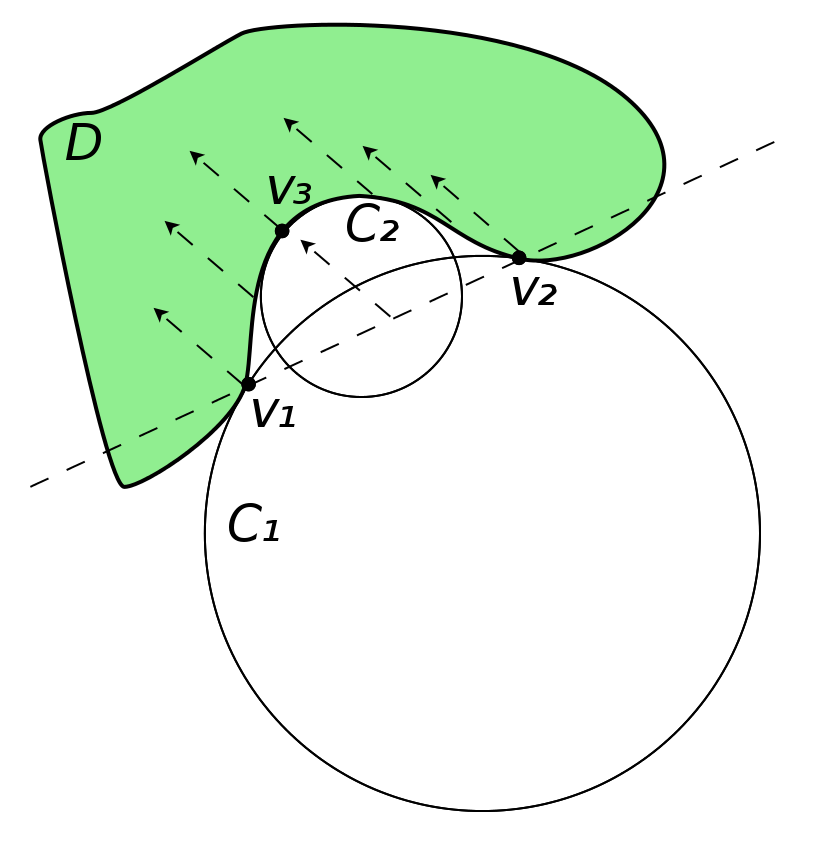}
\caption{Illustration of the argument showing that the uniform exterior sphere and interior cone conditions
imply positive reach.
}
\label{fig:positive-reach}
\end{figure}
%
Consider now the case in which the disk \(C_2\) intersects \(\ell\) at two points. Again using the Lipschitz representation,
the intersections must lie on the segment \(v_1v_2\) (since both \(v_1\) and \(v_2\) are 
in \(\prt D\cap M\)).
 The configuration of \(C_1\), \(C_2\), \ldots is illustrated in Figure \ref{fig:positive-reach}.
The planar geometry of circles allows us to deduce that the radius \(r\sin\alpha\)  of \(C_2\) must be smaller than the radius of \(C_1\), which itself is no larger than \(s\). Hence, we deduce that \(s\geq r\sin\alpha\) in this case as well.
\end{proof}


%

We quote the following theorem verbatim from \citet[Theorem 1.1]{Lytchak-2004}.
\begin{thm}\label{m29.3}
Let $M$ be a smooth Riemannian manifold, $Z$ a compact subset of $M$
that has positive reach. Then $Z$ has an upper curvature bound with respect to the
inner metric.
\end{thm}


Lemma \ref{n7.1} and Theorem \ref{m29.3} imply immediately

\begin{cor}\label{n7.2}
Suppose that $D$ is a bounded domain which satisfies the uniform exterior sphere
and uniform interior cone conditions. 
Then $\ol D$ is a $\CAT\kappa$ space for some $\kappa \ge 0$. 
\end{cor}
\begin{rem}\label{m29.1}
\cite[Theorem 12]{AlexanderBishopGhrist-2009} give a short proof of this result in the case where \(D\) has a smooth boundary.
\end{rem}


%
\begin{example}\label{ex:two-examples}\rm
We give two examples illustrating concepts related to Lemma \ref{n7.1} and Corollary \ref{n7.2}.  The easy
proofs are left to the reader.

(i) Let $r_n = \frac13 (\frac1n - \frac1{n+1})$ and let  $C_n$ be the circle $\{(z_1,z_2,z_3) \in \Reals^3: (z_1-\frac1n )^2 + z_2^2 = r_n^2, z_3 = 0\}$. Let $\ball_n^+$ and $\ball_n^-$ be two distinct balls with radii 1 such that the intersection of their boundaries with the plane $\{(z_1,z_2,z_3) \in \Reals^3: z_3 = 0\}$ is $C_n$. Let $D = \ball(0,10) \setminus \bigcup_{n\geq 1} (\ol \ball_n^+ \cup \ol \ball_n^-)$.  
The domain $D$ satisfies a uniform exterior sphere
condition, but it does not have a positive reach.

(ii) Let $z_1=(0,0,1)$, $z_2 = (0,0,-1)$ and $D = \ball(0, 10) \setminus (\ol{\ball(z_1,1)} \cup \ol{\ball(z_2,1)})$. The domain $D$ satisfies a uniform exterior sphere
condition and has a positive reach,
but it is not a Lipschitz domain.
\end{example}

\end{document}